\documentclass[11pt]{amsart}
\newfont{\cyr}{wncyr10 scaled 1100}
\usepackage[top=2.8cm, bottom=2.8cm, left=2.6cm, right=2.6cm]{geometry}
\usepackage{amsthm,amssymb,amsmath,amsfonts,mathrsfs,amscd, graphics}
\usepackage{mathtools}
\usepackage[latin1]{inputenc}
\usepackage[all]{xy}
\usepackage{latexsym}
\usepackage{longtable}
\usepackage{color}
\usepackage{tikz}
\usepackage{tikz-cd}
\usepackage{lua-visual-debug}
\usepackage{dsfont}
\usetikzlibrary{matrix}
\usepackage{hyperref}
\usepackage{enumitem}
\usepackage{setspace}

\newcommand*\ZZ{|[draw,circle]| \Z_2}
\numberwithin{equation}{section}
\theoremstyle{plain}
\newtheorem{theorem}{Theorem}[section]
\newtheorem*{theorem*}{Theorem}
\newtheorem{corollary}[theorem]{Corollary}
\newtheorem{lemma}[theorem]{Lemma}
\newtheorem{proposition}[theorem]{Proposition}

\numberwithin{equation}{section}
\newtheorem{thm}{Theorem}
\newtheorem{ass}[thm]{Assumption}

\theoremstyle{definition}
\newtheorem{definition}[theorem]{Definition}

\theoremstyle{remark}
\newtheorem{obswr}[theorem]{Observation}
\newtheorem{remarkwr}[theorem]{Remark}

\newtheorem{intro-definition}[theorem]{Definition}

\newenvironment{remark}{\begin{remarkwr}\begin{upshape}}{\end{upshape}\end{remarkwr}}
\newenvironment{myproof}[2] {\paragraph{\emph{Proof of {#1} {#2} }}}{\hfill$\square$}

\def\Gal{\mathrm{Gal}}
\def\GL{\mathrm{GL}}

\def\det{\mathrm{det}}

\def\ord{\mathrm{ord}}
\def\Spec{\mathrm{Spec}}

\def\coker{\mathrm{coker}}

\def\sing{\mathrm{sing}}
\def\Frob{\mathrm{Frob}}
\def\new{\mathrm{new}}
\def\ur{\mathrm{ur}}
\def\ac{\mathrm{ac}}
\def\CH{\mathrm{CH}}

\def\Nrd{\mathrm{Nrd}}

\DeclareMathOperator{\End}{End}

\def\calA{\mathcal{A}}

\def\calF{\mathcal{F}}
\def\calH{\mathcal{H}}
\def\calL{\mathcal{L}}

\def\calM{\mathcal{M}}
\def\calO{\mathcal{O}}

\def\calW{\mathcal{W}}

\def\gothN{\mathfrak{N}}

\def\frakL{\mathfrak{L}}
\def\frakl{\mathfrak{l}}

\def\Adel{\mathbf{A}}

\def\CC{\mathbf{C}}

\def\FF{\mathbf{F}}

\def\PP{\mathbf{P}}
\def\QQ{\mathbf{Q}}
\def\RR{\mathbb{R}}
\def\TT{\mathbb{T}}

\def\ZZ{\mathbf{Z}}

\def\rmA{\mathrm{A}}

\def\rmH{\mathrm{H}}

\def\rmE{\mathrm{E}}

\def\rmM{\mathrm{M}}

\def\rmV{\mathrm{V}}

\def\rmT{\mathrm{T}}

\def\ac{\mathrm{ac}}
\def\frakm{\mathfrak{m}}

\def\Qbar{\QQ^{\ac}}

\def\interX{\mathfrak{X}}
\def\ss{\mathrm{ss}}

\include{thebibliography}

\begin{document}

\title[Heegner cycles]
{Indivisibility of Heegner cycles over Shimura curves and Selmer groups}

\author{Haining Wang }
\address{\parbox{\linewidth} {Haining Wang,\\ Department of Mathematics,\\ McGill University,\\ 805 Sherbrooke St W,\\ Montreal, QC H3A 0B9, Canada.~ }}
\email{wanghaining1121@outlook.com}

\begin{abstract}
In this article, we show that the Abel-Jacobi images of  the Heegner cycles over the Shimura curves constructed by Nekovar, Besser and the theta elements contructed by Chida-Hsieh form a bipartite Euler system in the sense of Howard. As an application of this, we deduce a converse to Gross-Zagier-Kolyvagin type theorem for higher weight modular forms generalizing works of Wei Zhang and Skinner for modular forms of weight two.  That is, we show if the rank of certain residual Selmer group is one, then the Abel-Jacobi image of the Heegner cycle is non-zero in this residual Selmer group. 
\end{abstract}

\date{\today}

\maketitle

\tableofcontents

\section{Introduction}

In a seminal work of Bertolini-Darmon \cite{BD-Main}, the authors constructed an Euler-Kolyvagin type system using Heegner points on various Shimura curves. The cohomology classes in this system satisfy beautiful reciprocity laws that resemble the so called Jochnowitz's congruences. More precisely, these reciprocity laws relate the theta elements of the Gross points on the Shimura set given by certain definite quaternion algebras to the reductions of the Heegner points on the Shimura curves given by certain indefinite quaternion algebras. These theta elements encode the algebraic part of the special values of the $L$-functions of elliptic curves over an imaginary quadratic field and the Heegner points provide natural classes in the (dual of the) Selmer groups for the elliptic curves over such imaginary quadratic field.   Therefore these reciprocity laws enabled the authors to construct annihilators for elements in the Selmer groups. As an application of these, the authors prove the one-sided divisibility of the anticyclotomic Iwasawa main conjecture for an elliptic curves. The method of Bertolini-Darmon is axiomatized in \cite{How} where it is shown that the theta elements and the Heegner points (almost) form a bipartite Euler system in his sense. See also the recent work \cite{BCK} for a refinement.

The present article addresses the question of constructing a bipartite Euler system for higher weight modular forms over an imaginary quadratic field. On the analytic side, the theta elements are constructed by Chida-Hsieh in \cite{CH-1}. On the geometric side, it is natural to consider the Heegner cycles constructed by Nekovar \cite{Nekovar-Heeg} over the classical modular curves and by Besser \cite{Bess}, Iovita-Speiss \cite{IS} over the Shimura curves given by indefinite quaternion algebras. In this article, we show that these Heegner cycles and the theta elements of Chida-Hsieh indeed form a bipartite Euler system. As an application of this, we prove a converse to Gross-Zaiger-Kolyvagin type  theorem which can be seen as the rank $1$ case of a generalization of the Kolyvagin's conjecture to higher weight modular forms. We follow the strategy of Wei Zhang in his proof of the original Kolyvagin's conjecture for modular forms of weight $2$. 

There are other attempts to generalize the work of Bertolini-Darmon \cite{BD-Main} to higher weight case. Notably in \cite{CH-2},  the authors indeed prove the one-sided divisibility for the anticyclotomic Iwasawa main conjecture for higher weight modular forms. Their construction relies on a clever trick using congruences between weight two modular forms and higher weight modular forms evaluated at Gross points. This strategy works well when the root number of the involved $L$-function is plus one but does not apply to questions when the root number is negative one. We also remark that in \cite{Chida}, the author works directly with the Heegner cycles but still in the case when the root is plus one and he is able to prove the first reciprocity law and apply it to prove a version of the Bloch-Kato conjecture in the rank zero case. In this article, we prove the remaining second reciprocity law which forms the main arithmetic input to our proof of the converse to the Gross-Zagier-Kolyvagin type theorem.

\subsection{Main results} To describe precisely our results, we first introduce some notations. Let $f\in S^{\new}_{k}(N)$ be a newform of level $\Gamma_{0}(N)$ with even weight $k$ and  $K$ be an imaginary quadratic field whose discriminant is given by $-D_{K}$ with $D_{K}>0$.  We assume that $N$ and $D_{K}$ are relatively prime to each other. We also assume that $N$ admits a factorization $N=N^{+}N^{-}$ with $N^{+}$ only divisible by primes that are split in $K$ and $N^{-}$ only divisible by primes that are inert in $K$. Throughout this article we assume that the following generalized Heegner hypothesis is satisfied:
\begin{equation*}\tag{Heeg}
\text{\emph{$N^{-}$ is square free and consists of even number of prime factors that are inert in $K$}}.
\end{equation*}
Let $l$ be a distinguished rational prime such that $l\nmid ND_{K}$ and $k<l-1$. Let $E=\QQ(f)$ be the Hecke field of $f$ and we fix an embedding $\iota_{l}: \QQ^{\ac}\hookrightarrow \CC_{l}$ such that it induces a place $\lambda$ of $E$. Let $E_{\lambda}$ be the completion of $E$ at $\lambda$ and $\calO=\calO_{E_{\lambda}}$ be the valuation ring of $E_{\lambda}$. We fix a uniformizer $\varpi\in\calO$ and let $\FF_{\lambda}$ be the residue field of $\calO$. If $n\geq 1$, then we will write $\calO_{n}=\calO/\varpi^{n}$. Let $\TT=\TT(N^{+}, N^{-})$ be the $l$-adic completion of the Hecke algebra acting faithfully on the subspace of $S_{k}(N)$ consisting of forms that are new at primes dividing $N^{-}$. Let $\phi_{f}: \TT\rightarrow \calO$ be the morphism corresponding to the Hecke eigensystem of $f$ and $\phi_{f, n}: \TT\rightarrow \calO_{n}$ be the reduction of $\phi_{f}$ modulo $\varpi^{n}$. Let $I_{f, n}$ be the kernel of $\phi_{f, n}$ and $\frakm_{f}$ be the unique maximal ideal containing $I_{f, n}$. We denote by
\begin{equation*}
\rho_{f, \lambda}: G_{\QQ}\rightarrow \GL_{2}(E_{\lambda})
\end{equation*}
the $\lambda$-adic Galois representation attached to $f$ whose residual Galois representation is denoted by $\bar{\rho}_{f, \lambda}$. In this article, we will consider the twist 
$\rho_{f, \lambda}(\frac{2-k}{2})$ which we will denote by $\rho^{*}_{f, \lambda}$ whose representation space is denoted by $V_{f, \lambda}$. We fix a $G_{\QQ}$-stable lattice $\rmT_{f, \lambda}$ in $V_{f, \lambda}$. The residual Galois representation of $\rho^{*}_{f, \lambda}$ will be denoted by $\overline{\rho}^{*}_{f, \lambda}$.  It is well-known that the representation $\rho^{*}_{f, \lambda}$ appears in the cohomology of certain Shimura curve with coefficient in some $l$-adic local system $\calL_{k-2}$ corresponding to the representation $\mathrm{Sym}^{k-2}\mathrm{st}\otimes \det^{\frac{k-2}{2}}$ of $\GL_{2}$ where $\mathrm{st}$ is the standard representation of $\GL_{2}$. To define these Shimura curves, we will introduce certain quaternion algebras.  Let $B^{\prime}$ be the indefinite quaternion algebra of discriminant $N^{-}$ and $\calO_{B^{\prime}, N^{+}}$ be an Eichler order of level $N^{+}$ contained in some maximal order $\calO_{B^{\prime}}$ of $B^{\prime}$. These data define a Shimura curve $X=X^{B^{\prime}}_{N^{+}, N^{-}}$ which is a coarse moduli space of abelian surfaces with quaternionic multiplication. We wish NOT to assume that $N^{-}>1$ in which case $X$ is a projective curve over $\QQ$.  In the case when $N^{-}=1$, $X$ will be the compactification of the classical modular curve over $\QQ$. However, we only give the constructions and proofs for the more complicated case of $N^{-}>1$. We will rigidify the moduli problem of $X$ by adding an auxiliary full $d$-level structure and denote the resulting fine moduli space by $X_{d}$. Let $ A_{d}\rightarrow X_{d}$ be universal abelian surface and $\pi_{k, d}: W_{k, d}\rightarrow X_{d}$ be the Kuga-Sato variety given by the $\frac{k-2}{2}$-fold fiber product of $A_{d}$ over $X_{d}$.  One can construct certain idempotent $\epsilon_{d}$ and $\epsilon_{k}$ that cuts out the motive of $f$ in the Kuga-Sato variety $W_{k,d}$. Then the representation $\rmT_{f, \lambda}$ occurs in $\epsilon_{d}\epsilon_{k}\rmH^{k-1}(W_{k,d, \QQ^{\ac}}, \calO(\frac{k}{2}))=\rmH^{1}(X_{\QQ^{\ac}}, \calL_{k-2}(\calO)(1))$.  We will put the following assumption on the residue Galois representation $\bar{\rho}_{f, \lambda}$.
\begin{ass}[$\mathrm{CR}^{\star}$]  The residual Galois representation $\bar{\rho}_{f, \lambda}$ satisfies the following assumptions.
\begin{enumerate}
\item $l>k+1$ and $|(\FF^{\times}_{l})^{k-1}|>5$;
\item $\bar{\rho}_{f,\lambda}$ is absolutely irreducible when restricted to $G_{\QQ(\sqrt{p^{*}})}$ where $p^{*}=(-1)^{\frac{p-1}{2}}p$;
\item If $q\mid N^{-}$ and $q\equiv \pm1\mod l$, then $\bar{\rho}_{f, \lambda}$ is ramified;
\item If $q\mid \mid N^{+}$ and $q\equiv 1\mod l$, then $\bar{\rho}_{f, \lambda}$ is ramified;
\item The Artin conductor $N_{\bar{\rho}}$ of $\bar{\rho}_{f, \lambda}$ is prime to $N/N_{\bar{\rho}}$;
\item There is a place $q\mid\mid N$ such that $\bar{\rho}_{f, \lambda}$ is ramified at $q$.
\end{enumerate}
\end{ass}
We remark that our assumption $(\mathrm{CR}^{\star})$ is essentially the assumption $(\mathrm{CR}^{+})$  in  \cite{CH-2}. It is used to invoke results in \cite{CH-1} and \cite{CH-2}. The assumption $(\mathrm{CR}^{\star}) (6)$ is needed to apply the main result of \cite{SU}.

Let $K_{m}$ be the ring class field over $K$ of level $m$ for some integer $m\geq 1$. As recalled in \S 3.1, we define certain Heegner cycle
$\epsilon_{d}Y_{m, k}\in  \epsilon_{d}\epsilon_{k}\CH^{\frac{k}{2}}(W_{k, d}\otimes K_{m})\otimes\ZZ_{l}$
and an Abel-Jacobi map for some $n\geq 1$
\begin{equation*}
\mathrm{AJ}_{k, n}: \epsilon_{d}\epsilon_{k}\CH^{\frac{k}{2}}(W_{k, d}\otimes K_{m})\otimes\ZZ_{l} \rightarrow  \rmH^{1}(K_{m}, \mathrm{T}_{f, n}).
\end{equation*}
The images of $\epsilon_{d}Y_{m, k}$ under the map $\mathrm{AJ}_{k, n}$ gives the cohomological class 
\begin{equation*}
\kappa_{n}(m):=\mathrm{AJ}_{k, n}(\epsilon_{d}Y_{m, k})\in  \rmH^{1}(K_{m}, \mathrm{T}_{f, n})
\end{equation*}
and we define 
\begin{equation*}
\kappa_{n}:=\mathrm{Cor}_{K_{1}/K}\kappa_{n}(1)\in  \rmH^{1}(K, \mathrm{T}_{f, n}).
\end{equation*}
Our main result concerns the element $\kappa_{1}$. The element $\kappa_{1}$ in fact lives in the residual Selmer group 
\begin{equation*}
\mathrm{Sel}_{\calF(N^{-})}(N^{+}, \rmT_{f, 1}) 
\end{equation*}
defined by some Selmer structures $\calF(N^{-})$ spelled out in \eqref{Selmer}. Our main result is the following.
\begin{thm}\label{main-theorem-intro}
Suppose $(f, K)$ is a pair that satisfies the generalized Heegner hypothesis $(\mathrm{Heeg})$ with $f$  ordinary at $l$ and that $\bar{\rho}_{f, \lambda}$ satisfies the hypothesis $(\mathrm{CR}^{\star})$. If $\dim_{\FF_{\lambda}}\mathrm{Sel}_{\calF(N^{-})}(K, \rmT_{f, 1})=1$,
then the class $\kappa_{1}$ is non-zero in  $\mathrm{Sel}_{\calF(N^{-})}(K, \rmT_{f, 1})$. 
\end{thm}
This theorem can be viewed as a converse to Gross-Zagier-Kolyvagin theorem for Heegner cycles. For the other direction, one can show $\kappa_{n}$ is non-zero in $\mathrm{Sel}_{\calF(N^{-})}(K, \rmT_{f, n})$, then the Selmer group $\mathrm{Sel}_{\calF(N^{-})}(K, \rmT_{f, n})$ is of rank $1$. This follows from the result of Nekovar \cite{Nekovar-Heeg} in the case when $N^{-}=1$ and its extension to the case when $N^{-}>1$ in \cite{EdVP}. In these works, they follow the method of Kolyvagin and uses the derivative classes of the Heegner cycles to construct annihilators 
for the Selmer groups. We can recover their theorems by combining the first and second reciprocity law proved in this article. See \cite{Wang2} for an example of how to carry this out. We also have the Gross-Zagier formula  \cite{zhang-Heegner} for the Heegner cycles over the classical modular curves by Shou-Wu Zhang.  Suppose that height pairing is non-degenerate, then the Gross-Zagier formula and Theorem \ref{main-theorem-intro} would allow us to conclude that if the rank of Selmer group is one, then the analytic rank of the $L$-function $L(f/K, \frac{k}{2})$ is one.  Next we sketch the proof of Theorem \ref{main-theorem-intro}. First we recall the notion of an $n$-admissible prime for $f$. 
\begin{definition}\label{adm}
We say a prime $p$  is $n$-admissible for $f$ if
\begin{enumerate}
\item $p\nmid Nl$;
\item $p$ is an inert in $K$;
\item $l$ does not divide $p^{2}-1$;
\item  $\varpi^{n}$ divides $p^{\frac{k}{2}}+p^{\frac{k-2}{2}}-\epsilon_{p}a_{p}(f)$ with $\epsilon_{p}\in\{\pm1\}$.
\end{enumerate}
\end{definition}
These primes are level raising primes for $f$. This is justified in the following theorem which we call the (unramified) arithmetic level raising theorem for the Kuga-Sato varieties. We consider the ordinary-supersingular excision exact sequence on $X$ with coefficient in $\calL_{k-2}$
\begin{equation*}
0\rightarrow\rmH^{1}(\overline{X}_{\FF^{\ac}_{p}}, \calL_{k-2}(\calO)(1))_{\frakm_{f}}\rightarrow \rmH^{1}(X^{\ord}_{\FF^{\ac}_{p}}, \calL_{k-2}(\calO)(1))_{\frakm_{f}}\rightarrow 
\rmH^{0}(X^{ss}_{\FF^{\ac}_{p}}, \calL_{k-2}(\calO))_{\frakm_{f}}\rightarrow0
\end{equation*}
whose coboundary map induces the following map
\begin{equation*}
\Phi_{n}: \rmH^{0}(X^{\ss}_{\FF^{\ac}_{p}}, \calL_{k-2}(\calO))^{G_{\FF_{p^{2}}}}_{/I_{f, n}}\rightarrow \rmH^{1}(\FF_{p^{2}},  \rmH^{1}(\overline{X}_{\FF^{\ac}_{p}}, \calL_{k-2}(\calO)(1))_{/I_{f, n}}).
\end{equation*}
\begin{thm}[Unramified level raising]\label{level-raise-curve}
Let $p$ be an $n$-admissible prime for $f$. We assume that the residual Galois representation $\bar{\rho}_{f,\lambda}$ satisfies $(\mathrm{CR}^{\star})$. Then the following holds true.
\begin{enumerate}
\item There exists a morphism $\phi^{[p]}_{f, n}: \TT^{[p]}\rightarrow \calO_{n}$ that agree with $\phi_{f, n}:\TT\rightarrow\calO_{n}$ on all the Hecke operators away from $p$ and sending $U_{p}$ to $\epsilon_{p}p^{\frac{k-2}{2}}$.

\item Let $I^{[p]}_{f, n}$ be the kernel of the morphism $\phi^{[p]}_{f, n}$. We have a canonical isomorphism 
\begin{equation*}
\Phi_{n}: \rmH^{0}(X^{\ss}_{\FF^{\ac}_{p}}, \calL_{k-2}(\calO))^{G_{\FF_{p^{2}}}}_{/I_{f, n}}\xrightarrow{\cong}  \rmH^{1}(\FF_{p^{2}}, \rmH^{1}(\overline{X}_{\FF^{\ac}_{p}}, \calL_{k-2}(\calO)(1))_{/I_{f, n}})
\end{equation*}
which can be identified with an isomorphism
\begin{equation*}
\Phi_{n}: S^{B}_{k}(N^{+}, \calO)_{/I^{[p]}_{f, n}}\xrightarrow{\cong}  \rmH^{1}(\FF_{p^{2}}, \rmH^{1}(\overline{X}_{\FF^{\ac}_{p}}, \calL_{k-2}(\calO)(1))_{/I_{f, n}}).
\end{equation*}
\end{enumerate}
\end{thm}
One can define a theta element $\Theta(f^{[p]}_{\pi^{\prime}})$ associated to the Jacquet-Langlands transfer $f^{[p]}_{\pi^{\prime}}$ of $f^{[p]}$ following Chida-Hsieh \cite{CH-1} that encodes the square root of the algebraic part of the $L$-function $L(f^{[p]}/K, \frac{k}{2})$. Note that the global root number of the $L$-function $L(f^{[p]}/K, s)$ at $\frac{k}{2}$ is $+1$. We have the following reciprocity formula relating the Heegner cycle class $\kappa_{n}$ to the theta element $\Theta(f^{[p]}_{\pi^{\prime}})$.

\begin{thm}[Second reciprocity law] 
Let $p$ be an $n$-admissible prime for $f$ and assume that $\bar{\rho}_{f, \lambda}$ satisfies assumption $(\mathrm{CR}^{\star})$. Let ${f}^{[p]}_{n}$ be a generator of $S^{B}_{k}(N^{+}, \calO)[I^{[p]}_{f, n}]$, then we have the following relation between the class $\kappa_{n}$ and the theta element $\Theta(f^{[p]}_{\pi^{\prime}})$
\begin{equation*}
\langle \mathrm{loc}_{p} (\kappa_{n}), f^{[p]}_{n}\rangle_{B}=u\cdot\Theta(f^{[p]}_{\pi^{\prime}}) \mod \varpi^{n}
\end{equation*}
for some unit $u\in \calO_{n}$. 
\end{thm}
Returning to the sketch of the proof of Theorem \ref{main-theorem-intro}, we choose an $1$-admissible prime $p$ for $f$ and consider the residual Selmer group $\mathrm{Sel}_{\calF(pN)}(K, \rmT_{f,1})$ associated to $f^{[p]}$. The assumption that the residual Selmer group $\mathrm{Sel}_{\calF(N^{-})}(K, \rmT_{f, 1})$ is of dimension $1$ ensures that the dimension of $\mathrm{Sel}_{\calF(pN)}(K, \rmT_{f,1})$ drops to $0$. As a consequence of the Iwasawa main conjectures for $f^{[p]}$ proved in \cite{SU} and \cite{CH-2}, we show that the algebraic part of the special value $L(f^{[p]}/K, \frac{k}{2})$ is indivisible by $\varpi$ and thus $\Theta(f^{[p]}_{\pi^{\prime}})$ is indivisible by $\varpi$. Here we will rely on the recent work of Kim-Ota \cite{OK} to compare the canonical period $\Omega^{\mathrm{can}}_{f^{[p]}}$ and another period $\Omega_{f^{[p]}, pN^{-}}$ that show up in the specialization formula relating $\Theta(f^{[p]}_{\pi^{\prime}})$ to $L(f^{[p]}/K, \frac{k}{2})$. Finally, the second reciprocity law implies that $\mathrm{loc}_{p} (\kappa_{1})$ is indivisible and therefore $\kappa_{1}$ is non-zero.

We finish this introduction with a few remarks on the related works. First of all, in \cite{wei-zhang}, the author proves the Kolyvagin's conjecture without assuming the rank of the Selmer group is $1$. It is reasonable to expect that one can formulate and prove an analogue of the Kolyvagin's conjecture for Heegner cycles using the result of the present article as the first step of an induction process. In this article, the derived classes of the Heegner cycles are completely untouched.  As pointed out by Francesc Castella, our results should also shed light to the Perrin-Riou's main conjecture for generalized Heegner cycles  given by \cite[Conjecture 5.1]{LV1}. Compare the proof of \cite[Proposition 3.7]{BCK} towards the original Perrin-Riou's main conjecture for Heegner points. In an unpublished work of Castella and Skinner, the authors carry out a similar program for the big Heegner point in the sense of Howard and it should be interesting to compare their results with the results in this article.

\subsection{Notations and conventions} We will use common notations and conventions in algebraic number theory and algebraic geometry. The cohomologies appeared in this article will be understood as the \'{e}tale cohomologies. For a field $K$, we denote by $K^{\ac}$ the separable closure of $K$ and let $G_{K}:=\Gal(K^{\ac}/K)$ to be the Galois group of $K$. We let $\Adel$ be the ring of ad\`{e}les over $\QQ$ and $\Adel^{(\infty)}$ be the subring of finite ad\`{e}les.  For a prime $p$, $\Adel^{(\infty, p)}$ denotes the prime-to-$p$ part of  $\Adel^{(\infty)}$. 

Let $F$ be a local field with ring of integers $\calO_{F}$ and residue field $k$. We let $I_{F}$ be the inertia subgroup of $G_{F}$. Suppose $\rmM$ is a $G_{F}$-module. Then the finite part $\rmH^{1}_{\mathrm{fin}}(F, \rmM)$ of $\rmH^{1}(F, \rmM)$ is defined to be $\rmH^{1}(k, \rmM^{I_{F}})$ and the singular part $\rmH^{1}_{\mathrm{sing}}(F, \rmM)$ of $\rmH^{1}(F, \rmM)$ is defined to be the quotient of $\rmH^{1}(F, \rmM)$ by the image of $\rmH^{1}_{\mathrm{fin}}(F, \rmM)$. 

We provide a list of quaternion algebras appearing in this article. Recall that $N^{-}$ is square free with even number of prime divisors and $p, p^{\prime}$  are $n$-admissible primes.

\begin{itemize}
\item $B^{\prime}$ is the indefinite quaternion algebra of discriminant $N^{-}$. 
\item $B$ is the definite quaternion algebra of discriminant $pN^{-}$. 
\item $B^{\prime\prime}$ is the indefinite quaternion algebra of discriminant $pp^{\prime}N^{-}$ 
\end{itemize}

\subsection*{Acknowledgements} We would like to thank Henri Darmon and Pengfei Guan for their generous support during this difficult time. We would like to thank Francesc Castella for useful communications. 

\section{Arithmetic level raising on Kuga-Sato varieties}

\subsection{Shimura curves and local system}
Let $N$ be a positive integer with a factorization $N=N^{+}N^{-}$ with $N^{+}$ and $N^{-}$ coprime to each other. We assume that $N^{-}$ is square-free and is a product of \emph{even} number of primes. Let $B^{\prime}$ be the indefinite quaternion algebra over $\QQ$ with discriminant $N^{-}$. Let $\calO_{B^{\prime}}$ be a maximal order of $B^{\prime}$ and  let $\calO_{B^{\prime}, N^{+}}$ be the Eichler order of level $N^{+}$ in $\calO_{B^{\prime}}$. We define $G^{\prime}$ to be the algebraic group over $\QQ$ given by $B^{\prime \times}$ and let $K^{\prime}_{N^{+}}$ be the open compact subgroup of $G^{\prime}(\mathbf{A}^{(\infty)})$ defined by $\widehat{\calO}^{\times}_{B^{\prime}, N^{+}}$. 
Let $X=X^{B^{\prime}}_{N^{+}, N^{-}}$ be the Shimura curve over $\QQ$ with level $K^{\prime}=K^{\prime}_{N^{+}}$. The complex points of this curve is given by the following double coset

\begin{equation*}
X(\CC)=G^{\prime}(\QQ)\backslash \calH^{\pm} \times G^{\prime}(\mathbf{A}^{(\infty)})/K^{\prime}.
\end{equation*}
We consider the functor $\mathfrak{X}$ on schemes over $\ZZ[1/N]$ which gives the following moduli problem. Let $S$ be a test scheme over $\ZZ[1/N]$, then $\mathfrak{X}(S)$ classifies the triples $(A, \iota, C)$ up to isomorphism where
\begin{enumerate}
\item $A$ is an $S$-abelian scheme of relative dimension $2$;
\item $\iota: \calO_{B^{\prime}}\hookrightarrow \End_{S}(A)$ is an embedding;
\item $C$ is an $\calO_{B^{\prime}}$-stable locally cyclic subgroup of $A[N^{+}]$ of order $(N^{+})^{2}$.
\end{enumerate}
It is well-known this moduli problem is coarsely representable by a projective scheme $\interX$ over $\ZZ[1/N]$ of dimension $1$. Let $\Lambda=\ZZ/l^{n}$ for some $n\geq 1$ or a finite extension of $\ZZ_{l}$. Then we define $\calL^{\prime}_{k-2}(\Lambda)$ to be the local system given by the composite map
\begin{equation}\label{local-system}
\pi^{\mathrm{alg}}_{1}(X)\rightarrow K^{\prime}\rightarrow \calO^{\times}_{B^{\prime}, l}\cong \GL_{2}(\ZZ_{l})\rightarrow \GL_{k+1}(\ZZ_{l})
\end{equation}
To rigidify the moduli problem $\interX$, we choose an auxiliary integer $d\geq 5$ that is prime to $Nl$ and add a full level-$d$-structure to the above moduli problem that is we add the following data to the above moduli problem: let
\begin{equation}\label{d-full-level}
\nu_{d}: (\calO_{B^{\prime}}/d)_{S}\rightarrow A[d]
\end{equation}
be an isomorphism of $\calO_{B^{\prime}}$-stable group schemes.  By forgetting the data $\nu_{d}$, we have natural map $c_{d}: \interX_{d}\rightarrow \interX$ which is Galois with covering group $G_{d}:=(\calO_{B^{\prime}}/d)^{\times}/\{\pm1\}$. Then this new moduli problem is representable by a projective scheme $\mathfrak{X}_{d}$   over $\ZZ[1/Md]$ of relative dimension $1$. We will set 
\begin{equation*}
K^{\prime}_{d}=\{g=(g_{v})_{v}\in K^{\prime}: g_{v}\equiv \begin{pmatrix}1&0\\0& 1\end{pmatrix} \mod v \text{ for all } v\mid d\}. 
\end{equation*}
We will denote by $X_{d}$ the base-change of $\mathfrak{X}_{d}$ to $\QQ$. Then the $\CC$-point of  $X_{d}$ is given by
\begin{equation*}
X_{d}(\CC)=G^{\prime}(\QQ)\backslash \calH^{\pm} \times G^{\prime}(\mathbf{A}^{\infty})/K^{\prime}_{d}.
\end{equation*}
Let $\pi_{d}: \calA_{d}\rightarrow \interX_{d}$ be the universal abelian surface. The sheaf $R^{1}\pi_{*}\Lambda$ over $\interX_{d}$ is equipped with an action of $\calO_{B^{\prime}_{l}}=\rmM_{2}(\ZZ_{l})$.
We then define the following local system on $\mathfrak{X}_{d}$
\begin{equation*}
\calL^{\prime}_{k-2}(\Lambda):=\mathrm{Sym}^{k-2}e\cdot R^{1}\pi_{*}\Lambda, \hphantom{aa} \calL_{k-2}(\Lambda):=\mathrm{Sym}^{k-2}e\cdot R^{1}\pi_{*}\Lambda(\frac{k-2}{2})
\end{equation*}
where $e$ is the idempotent given by the matrix $\begin{pmatrix}1& 0\\0 &0\\ \end{pmatrix}$ in $\rmM_{2}(\ZZ_{l})$. There is another construction of this local system in \cite{Bess} whose cohomology is tied more closely to the Kuga-Sato variety that we will introduce below.  We briefly review this construction. Define
\begin{equation}
\mathfrak{L}^{\prime}_{2}(\Lambda)=\cap_{b\in B^{\prime}}\ker[R^{2}\pi_{k*}\Lambda\xrightarrow{b-\Nrd(b)} R^{2}\pi_{k*}\Lambda].
\end{equation}
We define the weight $k-2$ local system by
\begin{equation*}
\frakL^{\prime}_{k-2}(\Lambda):=\ker[\mathrm{Sym}^{m}\frakL_{2}(\Lambda)\xrightarrow{\Delta_{m}}\mathrm{Sym}^{m-2}\frakL_{2}(\Lambda)(-2)]
\end{equation*}
where $m=\frac{k-2}{2}$ and $\Delta_{m}$ is the Lapalace defined by 
\begin{equation*}
\Delta_{m}(x_{1}, \cdots, x_{m})=\sum_{1\leq i, j\leq m}(x_{i}, x_{j})x_{1}\cdots \hat{x}_{i}\cdots \hat{x}_{j} \cdots x_{m}
\end{equation*}
where $(\hphantom{a}, \hphantom{a})$ is the non-degenerate pairing 
\begin{equation*}
(\hphantom{a}, \hphantom{b}): \mathfrak{L}^{\prime}_{2}(\Lambda)\times \mathfrak{L}^{\prime}_{2}(\Lambda) \rightarrow \Lambda(-2)
\end{equation*}
induced by the Poincare duality
\begin{equation*}
(\hphantom{a}, \hphantom{b}): R^{2}\pi_{*}\Lambda \times  R^{2}\pi_{*}\Lambda \rightarrow \Lambda(-2).
\end{equation*}

\begin{lemma}
We have an isomorphism 
\begin{equation*}
\frakL^{\prime}_{k-2}(\Lambda)\cong \calL^{\prime}_{k-2}(\Lambda).
\end{equation*}
\end{lemma}
\begin{proof}
This result is contained in the proof of \cite[Theorem 5.8]{Bess}. We briefly outline the construction. First of all, there is an isomorphism between the rank $3$ local system
\begin{equation*}
\frakL^{\prime}_{2}(\Lambda)\cong \mathrm{Sym}^{2} e\cdot R^{1}\pi_{*}\Lambda.
\end{equation*}
Then it is an easy exercise to show that $\mathrm{Sym}^{k-2} e\cdot R^{1}\pi_{*}\Lambda$ is the kernel of the map
\begin{equation*}
\mathrm{Sym}^{m} \mathrm{Sym}^{2} e\cdot R^{1}\pi_{*}\Lambda \xrightarrow{\Delta_{m}} \mathrm{Sym}^{m-2} \mathrm{Sym}^{2} e\cdot R^{1}\pi_{*}\Lambda. 
\end{equation*}
 with $m=\frac{k-2}{2}$. The result follow from this. 
\end{proof}

Let $k\geq 2$ and let $\pi_{k, d}:\calW_{k, d}\rightarrow \interX_{d}$ be the \emph{Kuga-Sato variety of weight $k$} over $\interX_{d}$. This is defined by the $\frac{k-2}{2}$-fold fiber product of $\calA_{d}$ over $\interX_{d}$
\begin{equation*}
\calW_{k, d}:=\calA_{d}\times_{\mathfrak{X}_{d}}\calA_{d} \cdots \times_{\mathfrak{X}_{d}} \calA_{d}.
\end{equation*}
We will denote by $W_{k,d}$ the scheme $\calW_{k, d}\otimes \QQ$. We define $\epsilon_{d}$ to be the projector given by  
\begin{equation*}
\epsilon_{d}=\frac{1}{\vert G_{d}\vert}\sum_{g\in G_{d}}g. 
\end{equation*}
Then the following relation clearly holds:
\begin{equation*}
\epsilon_{d}\rmH^{1}(\mathfrak{X}_{d}\otimes\QQ^{\ac}, \calL^{\prime}_{k-2}(\Lambda))\cong \rmH^{1}(\mathfrak{X}\otimes\QQ^{\ac}, \calL^{\prime}_{k-2}(\Lambda)).
\end{equation*}

The cohomology of the Kuga-Sato variety and the cohomology of the local system $\calL^{\prime}_{k-2}(\Lambda)$ are closely related. 
\begin{lemma}\label{vanishing}
There is a projector $\epsilon_{k}$ on $\calW_{k, d}$ such that
\begin{equation*}
\epsilon_{k}\rmH^{*}(\calW_{k, d}\otimes\QQ^{\ac}, \Lambda)\cong \epsilon_{k}\rmH^{k-1}(\calW_{k, d}\otimes\QQ^{\ac}, \Lambda)\cong \rmH^{1}(\mathfrak{X}_{d}\otimes\QQ^{\ac}, \calL_{k-2}(\Lambda)).
\end{equation*}
\end{lemma}
\begin{proof}
This follows from the discussions in \cite[(67), Lemma 10.1]{IS}.
\end{proof}

\subsection{Shimura sets and quaternionic modular forms} Let $k\geq 2$ be an even integer such that $k<l-1 $. If $A$ is a ring, let $L_{k-2}(A)=\mathrm{Sym}^{k-2}(A^{2})$ be the set of homogeneous polynomials of degree $k-2$ with coefficient in $A$. We present $L_{k-2}(A)$ as
\begin{equation}
L_{k-2}(A)=\bigoplus_{\frac{-k}{2}\leq r\leq \frac{k}{2}}A\cdot\mathbf{v}_{r}
\end{equation}
with $\mathbf{v}_{r}:=X^{\frac{k-2}{2}-r}Y^{\frac{k-2}{2}+r}$.  It gives rise to a unitary representation
\begin{equation*}
\rho_{k}: \GL_{2}(A)\rightarrow \mathrm{Aut}_{A}(L_{k-2}(A))
\end{equation*}
such that $\rho_{k}(g)P(X, Y)=\det^{-\frac{(k-2)}{2}}(g)P((X, Y)g)$ for any $P(X, Y)\in L_{k}(A)$. Let $A$ be a $\ZZ_{(l)}$-algebra, we define a pairing 
\begin{equation*}
\langle\hphantom{a},\hphantom{v} \rangle: L_{k-2}(A)\times L_{k-2}(A)\rightarrow A 
\end{equation*}
by the following formula
\begin{equation*}
\langle\sum_{i}a_{i}\mathbf{v}_{i}, \sum_{j}b_{j}\mathbf{v}_{j}\rangle_{k-2}= \sum_{\frac{-k}{2}\leq r\leq \frac{k}{2}} a_{r}b_{-r}\cdot(-1)^{\frac{k-2}{2}+r}\frac{\Gamma(\frac{k}{2}+r)\Gamma(\frac{k}{2}-r)}{\Gamma(k-1)}.
\end{equation*}
For $P_{1}, P_{2}\in L_{k-2}(A)$, the pairing above has the following property
\begin{equation*}
\langle\rho_{k}(g)P_{1}, \rho_{k}(g)P_{2}\rangle=\langle P_{1}, P_{2} \rangle.
\end{equation*}
Let $p\nmid N$ be a prime. Let $B$ be the definite quaternion algebra of discriminant $pN^{-}$.  We let $G$ be the algebraic group over $\QQ$ defined by $B^{\times}$. If $U\subset G(\Adel^{(\infty)})$ is an open compact subgroup and $A$ is a $\ZZ_{l}$-algebra, we define the space $S^{B}_{k}(U, A)$ of $l$-adic quaternionic modular forms of weight $k$ with value in $A$ by
\begin{equation*}
S^{B}_{k}(U, A)=\{h: G(\Adel)\rightarrow L_{k-2}(A): h(agu)=\rho_{k}(u_{l}^{-1})h(g)\text{ for $a\in B^{\times}$ and $u\in U\cdot Z(\Adel^{(\infty)})$}\}.
\end{equation*}
In the case $U$ corresponds to an Eichler order $\calO_{B, N^{+}}$ of level $N^{+}$ in a fixed maximal order $\calO_{B}$, then we will simply write the space $S^{B}_{k}(U, A)$ as $S^{B}_{k}(N^{+}, A)$.  We will define an inner product on this space 
\begin{equation}\label{pairing}
\langle\hphantom{a}, \hphantom{b}\rangle_{B}: S^{B}_{k}(N^{+}, A)\times S^{B}_{k}(N^{+}, A) \rightarrow A
\end{equation}
by the following formula
\begin{equation}
\langle f_{1}, f_{2}\rangle_{B}=\sum_{g\in \mathrm{Cl(N^{+})}}\frac{1}{\vert\Gamma_{g}\vert}\langle f_{1}(g),f_{2}(g\tau^{N^{+}})\rangle_{k}
\end{equation}
where $\Gamma_{g}=(B^{\times}\cap g\widehat{\calO}_{B, N^{+}}^{\times}g^{-1}Z(\Adel^{(\infty)}))/\QQ^{\times}$ and $\mathrm{Cl(N^{+})}$ is a set of representatives of the 
\begin{equation*}
B^{\times}\backslash \widehat{B}^{\times}/\widehat{\calO}^{\times}_{B, N^{+}}\widehat{\QQ}^{\times}.
\end{equation*}

\subsection{Reductions of Shimura curves}
Let $p$ be a prime away from $N$. We will consider the base-change of $\interX_{d}$, $\interX$ to $\ZZ_{p^{2}}$ and we will denote them by the same notations. The special fiber  of $\interX_{d}$ and $\interX$ will be denoted by $\overline{X}_{d}$ and  $\overline{X}$. Let $x=(A, \iota, \bar{\eta})\in \overline{X}_{d}(\FF^{\ac}_{p})$ be an $\FF^{\ac}_{p}$-point. Then the $p$-divisible group $A[p^{\infty}]$ of $A$ can be written as  $A[p^{\infty}]=E[p^{\infty}]\times E[p^{\infty}]$ for a $p$-divisible group $E[p^{\infty}]$ associated to an elliptic curve $E$ and $\calO_{B^{\prime}}$ acts naturally via $\calO_{B^{\prime}}\otimes \ZZ_{p}=\rmM_{2}(\ZZ_{p})$. Depending on $E[p^{\infty}]$ is \emph{ordinary} or \emph{supersingular}, we will accordingly call $x$ ordinary or supersingular. Let $X^{\ss}_{d}$ be the closed sub-scheme of $\overline{X}_{d}$ given by those points that are supersingular and let $X^{\ord}_{d}=\overline{X}_{d}-X^{\ss}_{d}$ be its complement. We will refer to $X^{\ss}_{d}$ as the \emph{supersingular locus} and to $X^{\ord}_{d}$ as the \emph{ordinary locus}.  Let $B=B_{pN^{-}}$ be the definite quaternion algebra with discriminant $pN^{-}$ and $\calO_{B}$ be a maximal order. Note that we can naturally view $K^{\prime(p)}_{d}$, the prime-to-$p$ part of $K^{\prime}_{d}$, as an open compact subgroup of $G(\Adel^{(\infty, p)})=B^{\times}(\Adel^{(\infty, p)})$. The scheme $X^{\ss}_{d}$ is given by a finite set of points and we have the following parametrization of it. 

\begin{lemma}
We have an isomorphism
\begin{equation*}
X^{\ss}_{d}\cong B^{\times}(\QQ)\backslash B^{\times}(\Adel^{(\infty)})/ K^{\prime(p)}_{d}\calO^{\times}_{B_{p}}.
\end{equation*}
\end{lemma}
\begin{proof}
The lemma is well-known and can be proved using essentially the same method of the classical work Deuring and Serre. See \cite[Lemma 9]{DT} for example.
\end{proof}
We will write 
\begin{equation}\label{Shi-set}
X^{B}_{d}=B^{\times}(\QQ)\backslash B^{\times}(\Adel^{(\infty)})/ K^{\prime(p)}_{d}\calO^{\times}_{B_{p}}
\end{equation}
and refer to it as the \emph{Shimura set} associated to the definite quaternion algebra $B$ with level $\Gamma_{0}(N^{+})\cap\Gamma(d)$. Therefore the above lemma can be rephrased as an isomorphism
\begin{equation*}
X^{\ss}_{d}\cong X^{B}_{d}.
\end{equation*}
Let $\calO_{B^{\prime}, pN^{+}}$ be an Eichler order of level $pN^{+}$ and let $K^{\prime}(p)$ be the associated open compact subgroup in $G^{\prime}(\Adel^{\infty})$. Similarly as in \eqref{d-full-level}, we define the open compact subgroup $K^{\prime}_{d}(p)$ of $G^{\prime}(\Adel^{\infty})$ by adding a full level $d$-structure to $K^{\prime}(p)$.  We have the curve $X_{d}(p)$ over $\QQ$ whose complex points are given by
\begin{equation*}
X_{d}(p)(\CC)=G^{\prime}(\QQ)\backslash \calH^{\pm} \times G^{\prime}(\Adel^{\infty})/K^{\prime}_{d}(p).
\end{equation*}
We define an integral model $\interX_{d}(p)$ over $\ZZ[1/dN]$ which represents the following functor.  Let $S$ be a test scheme over $\ZZ[1/N]$. Then $\interX_{d}(p)(S)$ classifies the tuples $(A_{1}, A_{2}, \iota_{1}, \iota_{2}, \pi_{A}, C, \nu_{d})$ up to isomorphism where
\begin{enumerate}
\item $A_{i}$ for $i=1, 2$ is an $S$-abelian scheme of relative dimension $2$;
\item $\iota_{i}: \calO_{B^{\prime}}\hookrightarrow \End_{S}(A_{i})$ is an action of $\calO_{B^{\prime}}$ on $\End_{S}(A_{i})$ for $i=1, 2$;
\item $\pi_{A}: A_{1}\rightarrow A_{2}$ is an isogeny of degree $p$ that commutes with the action of $\calO_{B^{\prime}}$;
\item $C$ is an $\calO_{B^{\prime}}$-stable locally cyclic subgroup of $A[N^{+}]$ of order $N^{+2}$;
\item $\nu_{d}: (\calO_{B^{\prime}}/d)_{S}\rightarrow A_{1}[d]$ is an isomorphism of $\calO_{B^{\prime}}$-stable group schemes. 
\end{enumerate}
By forgetting the data given by $\nu_{d}$, we have natural map $c_{d}(p): \interX_{d}(p)\rightarrow\interX_{0}(p)$ where $\interX_{0}(p)$ is the coarse moduli space representing the above functor without the data $\nu_{d}$. Note that the isogeny $\pi_{A}$ induces an isomorphism between $\gamma_{d}:A_{1}[d]\xrightarrow{\sim}A_{2}[d]$ and $\gamma_{N^{+}}:A_{1}[N^{+}]\xrightarrow{\sim}A_{2}[N^{+}]$. Again we consider the base-change of $\interX_{d}(p)$ to $\ZZ_{p^{2}}$ and use the same symbol for this base-change and denote its special fiber by $\overline{X}_{d}(p)$. Let $\overline{X}_{0}(p)$ be the image of $\overline{X}_{d}(p)$ under the map $c_{d}(p)$. We have the following descriptions of $\overline{X}_{d}(p)$ and $\overline{X}_{0}(p)$. Similarly let  $(\overline{X}, X^{\ord}, X^{\ss})$ be the image of  $(\overline{X}_{d}, X^{\ord}_{d}, X^{\ss}_{d})$ under the map $c_{d}$, 
we will call $(X^{\ord}, X^{\ss})$ the ordinary and supersingular locus of $\overline{X}$.

\begin{lemma}
The scheme $\overline{X}_{d}(p)$ consists of two irreducible components both isomorphic to $\overline{X}_{d}$ which cross transversally at the supersingular locus ${X}^{ss}_{d}(p)$ of $\overline{X}_{d}(p)$ which can be identified with the supersingular locus of $\overline{X}_{d}$. A similar statement holds for $\overline{X}_{0}(p)$. 
\end{lemma}

\begin{proof}
This is proved in \cite[Theorem 4.7(v)]{Buzzard}.
\end{proof}

Let $\pi_{1}:\interX_{d}(p)\rightarrow \interX_{d}$ be the morphism given by sending $(A_{1}, A_{2}, \pi_{A}, \iota_{1}, \iota_{2}, C, \nu_{d})$ to $(A_{1}, \iota_{1}, C, \nu_{d})$ and  $\pi_{2}:\interX_{d}(p)\rightarrow \interX_{d}$ be the morphism given by sending $(A_{1}, A_{2}, \pi_{A}, \iota_{1}, \iota_{2}, C, \nu_{d})$ to $$(A_{2}, \iota_{2}, \gamma_{N^{+}}(C), \gamma_{d}\circ\nu_{d}).$$ Then we can define two closed immersions $i_{1}:\overline{X}_{d}\rightarrow \overline{X}_{d}(p)$ and $i_{2}: \overline{X}_{d}\rightarrow \overline{X}_{d}(p)$ as in the proof of \cite[Theorem 4.7(v)]{Buzzard} such that 
\begin{equation}\label{inter-matrix}
\begin{pmatrix}
\pi_{1}\circ i_{1} &   \pi_{1}\circ i_{2}\\
\pi_{2}\circ i_{1} &  \pi_{2}\circ i_{2}\\ 
\end{pmatrix}
=
\begin{pmatrix}
\mathrm{id} &   \Frob_{p}\\
S^{-1}_{p}\Frob_{p} & \mathrm{id}\\ 
\end{pmatrix}.
\end{equation}
where $S_{p}$ corresponds to the central element in the spherical Hecke algebra of $\GL_{2}(\QQ_{p})$.

We will need the following important result known as  the ``Ihara's lemma". This is proved for the case of classical modular curves by Ribet \cite{Ri-rising} and Diamond-Taylor \cite{DT} for Shimura curves.
\begin{theorem}[{\cite[Theorem 4]{DT}}] \label{DT} We have the following statements
\begin{enumerate}
\item The kernel of the pull-back map
\begin{equation*}
(\pi^{*}_{1}+\pi^{*}_{2}): \rmH^{1}(\interX_{d}\otimes{\Qbar}, \calL_{k-2}(\FF_{l}))\oplus \rmH^{1}(\interX_{d}\otimes{\Qbar}, \calL_{k-2}(\FF_{l}))\rightarrow \rmH^{1}(\interX_{d}(p)\otimes{\Qbar},\calL_{k-2} (\FF_{l}))
\end{equation*}
is Eisenstein.
\item The cokernel of the push-forwad map
\begin{equation*}
(\pi_{1*}, \pi_{2*}): \rmH^{1}(\interX_{d}(p)\otimes{\Qbar}, \calL_{k-2}(\FF_{l}))\rightarrow \rmH^{1}(\interX_{d}\otimes{\Qbar},  \calL_{k-2}(\FF_{l}))\oplus \rmH^{1}(\interX_{d}\otimes{\Qbar},  \calL_{k-2}(\FF_{l}))
\end{equation*}
is Eisenstein.
\end{enumerate}
\end{theorem}
\begin{remark}
Here the restriction on weights in the original treatment of  \cite{DT} can be improved by using the recent work of \cite{MS-Ihara}.  However, we will rely on the freeness result on the Hecke module of quaternionic modular forms \cite[Proposition 6.8]{CH-1} and this result is proved under the weight restriction in the Fontaine-Laffaille range. 
\end{remark}

\subsection{Review of weight spectral sequence} Let $K$ be a henselian discrete valuation field with valuation ring $\calO_{K}$ and residue field $k$ of characteristic $p$. We fix a uniformizer $\pi$ of $\calO_{K}$. We set $S=\Spec(\calO_{K})$, $s=\Spec(k)$ and $\eta=\Spec(K)$. Let ${K^{\ac}}$ be a separable closure of $K$ and $K_{\ur}$ the maximal unramified extension of $K$ in $K^{\ac}$. We denote by $k^{\ac}$ the residue field of $K_{\ur}$. 
Let $I_{K}=\Gal(K^{\ac}/K_{\ur})\subset G_{K}=\Gal(K^{\ac}/K)$ be the inertia subgroup. Let $l$ be a prime different from $p$. We set $t_{l}: I_{K}\rightarrow \ZZ_{l}(1)$ to be the canonical surjection given by 
\begin{equation*}
\sigma \mapsto (\sigma(\pi^{1/l^{m}})/\pi^{1/l^{m}})_{m}
\end{equation*} 
for every $\sigma\in I_{K}$. 

Let $\mathfrak{X}$ be a \emph{strict semi-stable scheme} over $S$ purely of relative dimension $n$ which we also assume to be proper. This means that $\mathfrak{X}$ is locally of finite presentation and Zariski locally \'{e}tale over $$\Spec(\calO_{K}[X_{1}, \cdots, X_{n}]/(X_{1}\cdots X_{r}-\pi))$$ for some integer $1\leq r\leq n$. We let $X_{k}$ be the special fiber of $\mathfrak{X}$ and $X_{k^{\ac}}$ be its base-change to $k^{\ac}$. Let $X=\mathfrak{X}_{\eta}$ be the generic fiber of $\mathfrak{X}$ and $X_{K_{\ur}}$ be its base-change to $K_{\ur}$. We have the following natural maps
$i:X_{k}\rightarrow \mathfrak{X}$,  $j: X\rightarrow \mathfrak{X}$, $\bar{i}: X_{k^{\ac}}\rightarrow \mathfrak{X}_{\calO_{K_{\ur}}}$ and  $\bar{j}: X_{K_{\ur}}\rightarrow \mathfrak{X}_{\calO_{K_{\ur}}}$.  
We have the \emph{Nearby cycle sheaf} given by 
\begin{equation*}
R^{q}\Psi(\Lambda)= \bar{i}^{*}R^{q}\bar{j}_{*}\Lambda
\end{equation*}
and the \emph{Nearby cycle complex} given by 
\begin{equation*}
R\Psi(\Lambda)= \bar{i}^{*}R\bar{j}_{*}\Lambda.
\end{equation*}
By proper base-change, we have $\rmH^{*}(X_{k^{\ac}}, R\Psi(\Lambda))=\rmH^{*}(X_{K^{\ac}}, \Lambda)$. We can regard $R\Psi(\Lambda)$ as an object in the derived category $D^{+}(X_{k^{\ac}}, \Lambda[I_{K}])$ of sheaves of $\Lambda$-modules with continuous $I_{K}$-actions. Let $D_{1},\cdots, D_{m} $ be the set of irreducible components of $X_{k}$. For each index set $I\subset \{1, \cdots, m\}$ of cardinality $p$, we set $X_{I, k}=\cap_{i\in I} D_{i}$. This is a smooth scheme of dimension $n-p$. For $1\leq p \leq m-1$, we define
\begin{equation*}
X^{(p)}_{k}=\bigsqcup_{I\subset \{1, \cdots, m\}, \mathrm{Card}(I)=p+1} X_{I, k}
\end{equation*}
 and let $a_{p}: X^{(p)}_{k}\rightarrow X_{k}$ be the natural projection, we have $a_{p *}\Lambda=\wedge^{p+1}a_{0 *}\Lambda$.
 
 Let $T$ be an element in $I_{K}$ such that $t_{l}(T)$ is a generator of $\ZZ_{l}(1)$ then $T$ induces a nilpotent operator $T-1$ on $R\Psi(\Lambda)$. Let $N=(T-1)\otimes\breve{T}$ where $\breve{T}\in \ZZ_{l}(-1)$ is the dual of $t_{l}(T)$. Then with respect to this $N$, we have the \emph{monodromy filtration} $M_{\bullet}R\Psi(\Lambda)$ on $R\Psi(\Lambda)$ characterized by 
\begin{enumerate}
\item $M_{n}R\Psi(\Lambda)=0$ and $M_{-n-1}R\Psi(\Lambda)=0$;
\item $N: R\Psi(\Lambda)(1)\rightarrow R\Psi(\Lambda)$ sends $M_{r}R\Psi(\Lambda)(1)$ into $M_{r-2}R\Psi(\Lambda)$ for $r\in\ZZ$;
\item $N^{r}: Gr^{M}_{r}R\Psi(\Lambda)(r)\rightarrow Gr^{M}_{-r}R\Psi(\Lambda)$ is an isomorphism. 
\end{enumerate}

The monodromy filtration induces the \emph{weight spectral sequence}
\begin{equation}\label{wt-seq}
\rmE^{p,q}_{1}= \rmH^{p+q}(X\otimes{k^{\ac}}, Gr^{M}_{-p}R\Psi(\Lambda))\Rightarrow \rmH^{p+q}(X\otimes{k^{\ac}}, R\Psi(\Lambda))=\rmH^{p+q}(X\otimes{K^{\ac}}, \Lambda).
\end{equation}
The $\rmE_{1}$-term of this spectral sequence can be made explicit by
\begin{equation*}
\begin{aligned}
\rmH^{p+q}(X\otimes{k^{\ac}}, Gr^{M}_{-p}R\Psi(\Lambda))&=\bigoplus_{i-j=-p, i\geq0, j\geq0}\rmH^{p+q-(i+j)}(X^{(i+j)}_{k^{\ac}}, \Lambda(-i))\\
&=\bigoplus_{i\geq\mathrm{max}(0, -p)}\rmH^{q-2i}(X^{(p+2i)}_{k^{\ac}}, \Lambda(-i)).\\
\end{aligned}
\end{equation*}
This spectral sequence is first introduced by Rapoport-Zink in \cite{RZ} and thus is also known as the Rapoport-Zink spectral sequence. 

Let $\mathfrak{X}$ be a relative curve over $\calO_{K}$.  Then we can immediately calculate that
\begin{equation*}
\begin{aligned}
&Gr^{M}_{-1}R\Psi(\Lambda)=a_{1*}\Lambda[-1], \\
&Gr^{M}_{0}R\Psi(\Lambda)=a_{0*}\Lambda,  \\
&Gr^{M}_{1}R\Psi(\Lambda)=a_{1*}\Lambda[-1](-1). \\
\end{aligned}
\end{equation*}
The $\rmE_{1}$-page of the weight spectral sequence is thus given by
\begin{center}
\begin{tikzpicture}[thick,scale=0.9, every node/.style={scale=0.9}]
  \matrix (m) [matrix of math nodes,
    nodes in empty cells,nodes={minimum width=5ex,
    minimum height=5ex,outer sep=-5pt},
    column sep=1ex,row sep=1ex]{
                &      &     &     & \\
          2     &  \rmH^{0}(\mathfrak{X}\otimes k^{\ac}, a_{1*}\Lambda(-1)) &  \rmH^{2}(\mathfrak{X}\otimes k^{\ac}, a_{0*}\Lambda)  & & \\
          1     &       & \rmH^{1}(\mathfrak{X}\otimes k^{\ac}, a_{0*}\Lambda) &    & \\
          0     &    & \rmH^{0}(\mathfrak{X}\otimes k^{\ac}, a_{0*}\Lambda) &  \rmH^{0}(\mathfrak{X}\otimes k^{\ac}, a_{1*}\Lambda) &\\
    \quad\strut &   -1  &  0  &  1  & \strut \\};
\draw[thick] (m-1-1.east) -- (m-5-1.east) ;
\draw[thick] (m-5-1.north) -- (m-5-5.north) ;
\end{tikzpicture}
\end{center}
and it clearly degenerates at the $\rmE_{2}$-page.  We therefore have the monodromy filtration
\begin{equation*}
0\subset^{\rmE^{1,0}_{2}} M_{1}\rmH^{1}(\mathfrak{X}\otimes{K^{\ac}}, \Lambda)\subset^{\rmE^{0,1}_{2}} M_{0}\rmH^{1}(\mathfrak{X}\otimes{K^{ac}}, \Lambda)\subset^{\rmE^{-1,2}_{2}} M_{-1}\rmH^{1}(\mathfrak{X}\otimes{K^{\ac}},\Lambda)=\rmH^{1}(\mathfrak{X}\otimes{K^{\ac}},\Lambda)
\end{equation*}
with the graded pieces given by
\begin{equation}\label{grad-curve}
\begin{aligned}
&Gr^{M}_{-1}\rmH^{1}(\mathfrak{X}\otimes{K^{\ac}}, \Lambda)=\ker[\rmH^{0}(\mathfrak{X}\otimes k^{\ac}, a_{1*}\Lambda(-1))\xrightarrow{\tau} \rmH^{2}(\mathfrak{X}\otimes k^{\ac}, a_{0*}\Lambda)]\\
&Gr^{M}_{0}\rmH^{1}(\mathfrak{X}\otimes{K^{\ac}}, \Lambda)= \rmH^{1}(\mathfrak{X}\otimes k^{\ac}, a_{0*}\Lambda)\\
&Gr^{M}_{1}\rmH^{1}(\mathfrak{X}\otimes{K^{\ac}} \Lambda)= \coker[\rmH^{0}(\mathfrak{X}\otimes k^{\ac}, a_{0*}\Lambda)\xrightarrow{\rho} \rmH^{0}(\mathfrak{X}\otimes k^{\ac}, a_{1*}\Lambda)]\\
\end{aligned}
\end{equation}
where $\tau$ is the \emph{Gysin morphism} and $\rho$ is the \emph{restriction morphism}. 
Note that the monodromy action on $\rmH^{1}(X\otimes K^{\ac}, \Lambda(1))$ can be understood using the following commutative diagram
\begin{equation*}
\begin{tikzcd}
\rmH^{1}(\mathfrak{X}\otimes{K^{\ac}}, \Lambda(1)) \arrow[r] \arrow[d, "N"] & \ker[\rmH^{0}(\mathfrak{X}\otimes k^{\ac}, a_{1*}\Lambda)\xrightarrow{\tau} \rmH^{2}(\mathfrak{X}\otimes k^{\ac}, a_{0*}\Lambda(1))] \arrow[d, "N"] \\
\rmH^{1}(\mathfrak{X}\otimes{K^{\ac}}, \Lambda)                  & \coker[\rmH^{0}(\mathfrak{X}\otimes k^{\ac}, a_{0*}\Lambda)\xrightarrow{\rho} \rmH^{0}(\mathfrak{X}\otimes k^{\ac}, a_{1*}\Lambda)] \arrow[l]
\end{tikzcd}
\end{equation*}
In this case, we recover the \emph{Picard-Lefschetz formula} if we identify $\rmH^{0}(\interX\otimes k^{\ac}, a_{1*}\Lambda)$ with the \emph{vanishing cycles} $\bigoplus_{x}R\Phi(\Lambda)_{x}$ on $X_{{k^{\ac}}}$ where $x$ runs through the singular points $X^{(1)}_{k^{\ac}}$ on $X_{k^{\ac}}$. 
Let $\rmM$ be a $G_{K}$-module over $\Lambda$, then we have the following exact sequence of Galois cohomology groups
\begin{equation}\label{fin-sing}
0\rightarrow \rmH^{1}_{\mathrm{fin}}(K, \rmM)\rightarrow \rmH^{1}(K, \rmM)\xrightarrow{\partial_{p}} \rmH^{1}_{\sing}(K, \rmM)\rightarrow 0  
\end{equation}
where  $\rmH^{1}_{\mathrm{fin}}(K, \rmM)=\rmH^{1}(k, \rmM^{I_{K}})$
is called the \emph{unramified} or \emph{finite} part of the cohomology group $\rmH^{1}(K, \rmM)$ and $\rmH^{1}_{\mathrm{sing}}(K, \rmM)$ defined as the quotient of $\rmH^{1}(K, \rmM)$ by its finite part is called the \emph{singular quotient} of  $\rmH^{1}(K, \rmM)$. The natural quotient map $\rmH^{1}(K, \rmM)\xrightarrow{\partial_{p}} \rmH^{1}_{\sing}(K, \rmM)$ will be referred to as the \emph{singular quotient map}. The element $\partial_{p}(x)$ will be referred to as the singular residue of $x$ for  $x\in \rmH^{1}(K, \rmM)$. Let $\rmM=\rmH^{n}(X_{K^{\ac}},\Lambda(r))$ be the $r$-th twist of the middle degree cohomology of $X_{{K^{\ac}}}$. We need the following elementary lemma.
\begin{lemma}\label{fin-sing-lemma}
Let $\rmM=\rmH^{n}(X_{K^{\ac}},\Lambda(r))$, then we have 
\begin{equation*}
\rmH^{1}_{\mathrm{fin}}(K, \rmM)\cong \frac{\rmM^{I_{K}}}{(\Frob_{p}-1)},\hphantom{a}\rmH^{1}_{\sing}(K, \rmM)\cong (\frac{\rmM(-1)}{N\rmM})^{G_{k}}.
\end{equation*}
\end{lemma}
\begin{proof}
This is well-known. The details can be found for example \cite[Lemma 2.6]{Liu-cubic}. 
\end{proof}

For $\rmM=\rmH^{1}(X_{K^{\ac}}, \Lambda(1))$, we can use the Picard-Lefschetz formula to calculate $\rmH^{1}_{\sing}(K, \rmM)$, we have 
\begin{equation}\label{1-sing}
\begin{aligned}
\rmH^{1}_{\sing}(K, \rmM) &\cong (\frac{\rmM(-1)}{N\rmM})^{G_{k}}
\cong (\frac{\coker[\rmH^{0}(\interX\otimes k^{\ac}, a_{0*}\Lambda)\xrightarrow{\rho} \rmH^{0}(\interX\otimes k^{\ac}, a_{1*}\Lambda)]}{N\ker[\rmH^{0}(\interX\otimes k^{\ac}, a_{1*}\Lambda)\xrightarrow{\tau} \rmH^{2}(\interX\otimes k^{\ac}, a_{0*}\Lambda(1))]})^{G_{k}}.\\
\end{aligned}
\end{equation}
Composing the isomorphism \eqref{1-sing} with $\tau$,  we have 
\begin{equation}
\begin{aligned}
 \rmH^{1}_{\sing}(K, \rmM)&\cong(\frac{\coker[\rmH^{0}(\interX\otimes k^{\ac}, a_{0*}\Lambda)\xrightarrow{\rho} \rmH^{0}(\interX\otimes k^{\ac}, a_{1*}\Lambda)]}{N\ker[\rmH^{0}(\interX\otimes k^{\ac}, a_{1*}\Lambda)\xrightarrow{\tau} \rmH^{2}(\interX\otimes k^{\ac}, a_{0*}\Lambda(1)))]})^{G_{k}}\\
 &\cong \coker[\rmH^{0}(\interX\otimes k^{\ac}, a_{0*}\Lambda)\xrightarrow{\rho} \rmH^{0}(\interX\otimes k^{\ac}, a_{1*}\Lambda)\xrightarrow{\tau}  \rmH^{2}(\interX\otimes k^{\ac}, a_{0*}\Lambda(1))]^{G_{k}}.\\
\end{aligned}
\end{equation}

Next we consider the curve $\mathfrak{X}_{d}(p)$ over $\Spec(\ZZ_{p^{2}})$. Let $\calA_{d}(p)\rightarrow \mathfrak{X}_{d}(p)$ be the universal abelian surface over $\interX_{d}(p)$. Then we define the \emph{Kuga-Sato variety} of weight $k$ by the $\frac{k-2}{2}$-fold fiber product of $\calA_{d}(p)$ over $\mathfrak{X}_{d}(p)$ that is
\begin{equation}\label{KS-var}
\calW_{k, d}(p):=\calA_{d}(p)\times_{\mathfrak{X}_{d}(p)}\calA_{d}(p) \cdots \times_{\mathfrak{X}_{d}(p)} \calA_{d}(p).
\end{equation}
Then there is a semi-stable model $\widetilde{\calW}_{k, d}(p)$ constructed in \cite[Lemma 4]{Saito2} and the action of the idempotent $\epsilon_{k}$ extends naturally to the semi-stable model  $\widetilde{\calW}_{k, d}(p)$. Moreover the first page of the weight spectral sequence converges to 
\begin{equation*}
\epsilon_{k}\rmH^{k-1}(\widetilde{\calW}_{k, d}(p)\otimes \QQ_{p}^{\ac}, \Lambda(\frac{k-2}{2}))=\rmH^{1}(\interX_{d}(p)\otimes \QQ_{p}^{\ac}, \calL_{k-2}(\Lambda))
\end{equation*}
takes the following form by \cite[page 37]{Saito2}.
\begin{center}
\begin{tikzpicture}[thick,scale=0.9, every node/.style={scale=0.9}]
  \matrix (m) [matrix of math nodes,
    nodes in empty cells,nodes={minimum width=5ex,
    minimum height=5ex,outer sep=-5pt},
    column sep=1ex,row sep=1ex]{
                &      &     &     & \\
          2     &  \rmH^{0}(\overline{X}_{d}(p)_{\FF_{p}^{\ac}}, a_{1*}\calL_{k-2}(\Lambda)(-1)) &  \rmH^{2}(\overline{X}_{d}(p)_{\FF_{p}^{\ac}}, a_{0*}\calL_{k-2}(\Lambda))  & & \\
          1     &       & \rmH^{1}(\overline{X}_{d}(p)_ {\FF_{p}^{\ac}}, a_{0*}\calL_{k-2}(\Lambda)) &    & \\
          0     &    & \rmH^{0}(\overline{X}_{d}(p)_{\FF_{p}^{\ac}}, a_{0*}\calL_{k-2}(\Lambda)) &  \rmH^{0}(\overline{X}_{d}(p)_{\FF_{p}^{\ac}}, a_{1*}\calL_{k-2}(\Lambda)) &\\
    \quad\strut &   -1  &  0  &  1  & \strut \\};
\draw[thick] (m-1-1.east) -- (m-5-1.east) ;
\draw[thick] (m-5-1.north) -- (m-5-5.north) ;
\end{tikzpicture}
\end{center}
This means that the weight spectral sequence of the Kuga-Sato variety agree with the weight spectral sequence of the base curve with certain non-trivial coefficient. By further applying the projector $\epsilon_{d}$, we obtain the following first page of the weight spectral sequence converging to $$\rmH^{1}( \interX_{0}(p)\otimes \QQ_{p}^{\ac}, \calL_{k-2}(\Lambda)).$$
\begin{center}
\begin{tikzpicture}[thick,scale=0.9, every node/.style={scale=0.9}]
  \matrix (m) [matrix of math nodes,
    nodes in empty cells,nodes={minimum width=5ex,
    minimum height=5ex,outer sep=-5pt},
    column sep=1ex,row sep=1ex]{
                &      &     &     & \\
          2     &  \rmH^{0}(\overline{X}_{0}(p)_{ \FF_{p}^{\ac}}, a_{1*}\calL_{k-2}(\Lambda)(-1)) &  \rmH^{2}(\overline{X}_{0}(p)_{ \FF_{p}^{\ac}}, a_{0*}\calL_{k-2}(\Lambda))  & & \\
          1     &       & \rmH^{1}(\overline{X}_{0}(p)_{\FF_{p}^{\ac}}, a_{0*}\calL_{k-2}(\Lambda)) &    & \\
          0     &    & \rmH^{0}(\overline{X}_{0}(p)_{\FF_{p}^{\ac}}, a_{0*}\calL_{k-2}(\Lambda)) &  \rmH^{0}(\overline{X}_{0}(p)_{\FF_{p}^{\ac}}, a_{1*}\calL_{k-2}(\Lambda)) &\\
    \quad\strut &   -1  &  0  &  1  & \strut \\};
\draw[thick] (m-1-1.east) -- (m-5-1.east) ;
\draw[thick] (m-5-1.north) -- (m-5-5.north) ;
\end{tikzpicture}
\end{center}
Note that we can make it explicit for the terms in the above spectral sequence
\begin{enumerate}
\item $ \rmH^{0}(\overline{X}_{0}(p)_{\FF_{p}^{\ac}}, a_{1*}\calL_{k-2}(\Lambda))=\rmH^{0}(X^{\ss}_{\FF_{p}^{\ac}}, \calL_{k-2}(\Lambda))$;
\item  $\rmH^{1}(\overline{X}_{0}(p)_{\FF_{p}^{\ac}}, a_{0*}\calL_{k-2}(\Lambda))=\rmH^{1}(\overline{X}_{\FF_{p}^{\ac}}, \calL_{k-2}(\Lambda))\oplus \rmH^{1}(\overline{X}_{\FF_{p}^{\ac}}, \calL_{k-2}(\Lambda))$;
\item  $\rmH^{2}(\overline{X}_{0}(p)_{\FF_{p}^{\ac}}, a_{0*}\calL_{k-2}(\Lambda))$ is Eisenstein;
\item  $\rmH^{0}(\overline{X}_{0}(p)_{\FF_{p}^{\ac}}, a_{0*}\calL_{k-2}(\Lambda))$ is Eisenstein.
\end{enumerate}

\subsection{Unramified level raising on the Kuga-Sato varities} Let $f\in S^{\new}_{k}(N)$ be a newform of level $\Gamma_{0}(N)$ with even weight $k$ and Fourier expansion $f=\sum a_{n}(f)q^{n}$. We denote by $E=\QQ(f)$ the Hecke field of $f$. Let $\lambda$ be a place of $E$ over $l$ and $E_{\lambda}$ be the completion of $E$ at $\lambda$.  Let $\varpi$ be a uniformizer of the ring of integers $\calO:=\calO_{E_{\lambda}}$ of $E_{\lambda}$ and $\FF_{\lambda}$ be its residue field. We will set $\calO_{n}=\calO/\varpi^{n}$. Let $K$ be an imaginary quadratic field whose discriminant is $-D_{K}$ with $D_{K}>0$. We assume that $N$ admits a factorization $N=N^{+}N^{-}$ where $N^{+}$ consists of prime factors that are split in $K$ and $N^{-}$ consists of prime factors that are inert in $K$. Let $\rho_{f, \lambda}: G_{\QQ}\rightarrow \GL_{2}(E_{\lambda})$ be the $\lambda$-adic Galois representation attached to the form $f$ characterized by the fact that the trace of Frobenius at $p\nmid N$ agrees with $a_{p}(f)$ and the determinant of $\rho_{f, \lambda}$ is $\epsilon_{l}^{k-1}$ with $\epsilon_{l}$ the $l$-adic cyclotomic character. We shall consider the twist $\rho^{*}_{f, \lambda}=\rho_{f, \lambda}(\frac{2-k}{2})$. Let $\rmV_{f, \lambda}$ be the representation space for $\rho^{*}_{f, \lambda}$.  We normalize the construction of $\rho_{f, \lambda}$ such that it occurs in the cohomology $\rmH^{1}(X_{\QQ^{\ac}}, \calL_{k-2}(E_{\lambda})(\frac{k}{2}))$ and therefore  $\rho^{*}_{f, \lambda}$ occurs in the cohomology $\rmH^{1}(X_{\QQ^{\ac}}, \calL_{k-2}(E_{\lambda})(1))$. Let $\TT=\TT(N^{+}, N^{-})$ be the $l$-adic completion of the integral Hecke algebra that acts faithfully on the subspace of $S_{k}(N)$ consisting of forms that are new at $N^{-}$ and let $\TT^{[p]}=\TT(N^{+}, pN^{-})$ be the $l$-adic completion of the integral Hecke algebra that acts faithfully on the subspace of $S_{k}(pN)$ consisting of modular forms that are new at $pN^{-}$.  The modular form $f$ gives rise to a homomorphism $\phi_{f}: \TT\rightarrow \calO$ corresponding to the Hecke eigensystem of $f$. FOr $n\geq 1$, we define $\phi_{f, n}: \TT\rightarrow \calO_{n}:=\calO/\varpi^{n}$ be the natural reduction of $\phi_{f}$ by $\varpi^{n}$.  We will define $I_{f, n}$ to be the kernel of the morphism $\phi_{f, n}$ and by $\frakm_{f}$ the unique maximal ideal of $\TT$ containing $I_{f, n}$. We fix a $G_{\QQ}$-stable lattice $\rmT_{f, \lambda}$ in $\rmV_{f, \lambda}$ and denote by $\rmT_{f, n}$ the reduction $\rmT_{f,\lambda}\mod \varpi^{n}$. We will always assume that the residual Galois representation $\bar{\rho}_{f, \lambda}$ satisfies the following assumption.

\begin{ass}[$\mathrm{CR}^{\star}$]  The residual Galois representation $\bar{\rho}_{f,\lambda}$ satisfies the following assumptions
\begin{enumerate}
\item $l>k+1$ and $|(\FF^{\times}_{l})^{k-1}|>5$;
\item $\bar{\rho}_{f,\lambda}$ is absolutely irreducible when restricted to $G_{\QQ(\sqrt{p^{*}})}$ where $p^{*}=(-1)^{\frac{p-1}{2}}p$;
\item If $q\mid N^{-}$ and $q\equiv \pm1\mod l$, then $\bar{\rho}_{f, \lambda}$ is ramified;
\item If $q\mid \mid N^{+}$ and $q\equiv 1\mod l$, then $\bar{\rho}_{f, \lambda}$ is ramified;
\item The Artin conductor $N_{\bar{\rho}}$ of $\bar{\rho}_{f, \lambda}$ is prime to $N/N_{\bar{\rho}}$;
\item There is a place $q\mid\mid N$ such that $\bar{\rho}_{f, \lambda}$ is ramified at $q$.
\end{enumerate}
\end{ass}

We will now prove a level raising result for the modular form $f$. First, we recall the following notion of $n$-admissible prime for $f$. 
\begin{definition}\label{adm}
We say a prime $p$  is $n$-admissible for $f$ if
\begin{enumerate}
\item $p\nmid Nl$;
\item $p$ is an inert in $K$;
\item $l$ does not divide $p^{2}-1$;
\item  $\varpi^{n}$ divides $p^{\frac{k}{2}}+p^{\frac{k-2}{2}}-\epsilon_{p}a_{p}(f)$ with $\epsilon_{p}\in\{\pm1\}$.
\end{enumerate}
\end{definition}

We consider the special fiber  $\overline{W}_{k, d}$ of $\calW_{k, d, \ZZ_{p^{2}}}$ and define its supersingular locus by $W^{ss}_{k, d}:=\pi^{-1}_{k, d}({X}^{ss}_{d})$. Similarly we define the supersingular locus of $\calW_{k, d}(p)$ by $W^{ss}_{k, d}(p):=\pi^{-1}_{k, d}(p)({X}^{ss}_{d}(p))$. We will consider the following \emph{ordinary-supersingular} excision exact sequence 
\begin{equation*}
\rmH^{\frac{k}{2}}(\overline{W}_{k, d, \FF^{\ac}_{p}}, \Lambda(\frac{k}{2}))\rightarrow \rmH^{\frac{k}{2}}(\overline{W}_{k, d, \FF^{\ac}_{p}}- W^{ss}_{k, d, \FF^{\ac}_{p}}, \Lambda(\frac{k}{2}))\rightarrow 
\rmH^{\frac{k}{2}+1}_{W^{ss}_{k, d, \FF^{\ac}_{p}}}(\overline{W}_{k, d,\FF^{\ac}_{p}}, \Lambda(\frac{k}{2})).
\end{equation*}
Apply the projector $\epsilon_{k}$ and localize at the maximal ideal $\frakm_{f}$, we in fact obtain an exact sequence 
\begin{equation*}
0\rightarrow\rmH^{1}(\overline{X}_{d, \FF^{\ac}_{p}}, \calL_{k-2}(\Lambda)(1))_{\frakm_{f}}\rightarrow \rmH^{1}(X^{\ord}_{d, \FF^{\ac}_{p}}, \calL_{k-2}(\Lambda)(1))_{\frakm_{f}}\rightarrow 
\rmH^{0}(X^{ss}_{d,\FF^{\ac}_{p}}, \calL_{k-2}(\Lambda))_{\frakm_{f}}\rightarrow 0
\end{equation*}
This follows from Lemma \ref{vanishing} and the following equation
\begin{equation*}
\rmH^{\frac{k}{2}+1}_{W^{ss}_{k, d}}(\calW_{k, d}\otimes\FF^{\ac}_{p}, \Lambda(\frac{k}{2}))\cong\bigoplus_{x\in X^{ss}_{d}}\rmH^{\frac{k-2}{2}}(\calA^{\frac{k-2}{2}}_{x}, \Lambda(\frac{k-2}{2}))
\end{equation*}
and
\begin{equation*}
\bigoplus_{x\in X^{\ss}_{d}}\epsilon_{k}\rmH^{\frac{k-2}{2}}(\calA^{\frac{k-2}{2}}_{x}, \Lambda(\frac{k-2}{2}))\cong \bigoplus_{x\in X^{ss}_{d}}L_{k-2}(\Lambda)
\end{equation*}
by the definition of $\epsilon_{k}$. The ordinary-supersingular excision exact sequence induces the following connecting homomorphism 
\begin{equation*}
\Phi_{d}: \rmH^{0}(X^{\ss}_{d, \FF^{\ac}_{p}}, \calL_{k-2}(\Lambda))^{G_{\FF_{p^{2}}}}_{\frakm_{f}}\rightarrow \rmH^{1}(\FF_{p^{2}},  \rmH^{1}(\overline{X}_{d, \FF^{\ac}_{p}}, \calL_{k-2}(\Lambda)(1))_{\frakm_{f}}).
\end{equation*}
By applying the projector $\epsilon_{d}$ to the above map, we have 
\begin{equation*}
\Phi: \rmH^{0}(X^{\ss}_{\FF^{\ac}_{p}}, \calL_{k-2}(\Lambda))^{G_{\FF_{p^{2}}}}_{\frakm_{f}}\rightarrow \rmH^{1}(\FF_{p^{2}},  \rmH^{1}(\overline{X}_{\FF^{\ac}_{p}}, \calL_{k-2}(\Lambda)(1))_{\frakm_{f}}).
\end{equation*}
By further taking the quotient by $I_{f, n}$, we have 
\begin{equation*}
\Phi_{n}: \rmH^{0}(X^{\ss}_{\FF^{\ac}_{p}}, \calL_{k-2}(\Lambda))^{G_{\FF_{p^{2}}}}_{/I_{f, n}}\rightarrow \rmH^{1}(\FF_{p^{2}},  \rmH^{1}(\overline{X}_{\FF^{\ac}_{p}}, \calL_{k-2}(\Lambda)(1))_{/I_{f, n}}).
\end{equation*}
The following theorem is known as the arithmetic level raising for the Kuaga-Sato variety. 

\begin{theorem}\label{level-raise-curve}
Let $p$ be an $n$-admissible prime for $f$. We assume that the residual Galois representation $\bar{\rho}_{f,\lambda}$ satisfies $(\mathrm{CR}^{\star})$. Then we have the following statements.
\begin{enumerate}
\item There exists a morphism $\phi^{[p]}_{f, n}: \TT^{[p]}\rightarrow \calO_{n}$ that agree with $\phi_{f, n}:\TT\rightarrow\calO_{n}$ for all the Hecke operators away from $p$ and sending $U_{p}$ to $\epsilon_{p}p^{\frac{k-2}{2}}$. 

\item Let $I^{[p]}_{f, n}$ be the kernel of the morphism $\phi^{[p]}_{f, n}$ and $\frakm^{[p]}_{f}$ be the maximal ideal containing $I_{f, n}$.  There exists a modular form $f^{[p]}\in S^{\new}_{k}(pN)$ such that the morphism $\phi^{[p]}_{f, n}$ lifts to $f^{[p]}$. 

\item The Hecke module  $\rmH^{1}(\interX\otimes\QQ^{\ac}, \calL_{k-2}(\calO)(1))_{\frakm_{f}}$ is free of rank $2$ over $\TT_{\frakm_{f}}$ and the Hecke module $S^{B}_{k}(N^{+}, \calO)_{\frakm^{[p]}_{f}}$ is free of rank one over $\TT^{[p]}_{\frakm^{[p]}_{f}}$.

\item We have a canonical isomorphism 
\begin{equation}\label{first-isom}
\Phi_{n}: \rmH^{0}(X^{\ss}_{\FF^{\ac}_{p}}, \calL_{k-2}(\calO))^{G_{\FF_{p^{2}}}}_{/I_{f, n}}\xrightarrow{\cong}  \rmH^{1}(\FF_{p^{2}}, \rmH^{1}(\overline{X}_{\FF^{\ac}_{p}}, \calL_{k-2}(\calO)(1))_{/I_{f, n}})\end{equation}
which can be identified with an isomorphism
\begin{equation}\label{second-isom}
\Phi_{n}: S^{B}_{k}(N^{+}, \calO)_{/I^{[p]}_{f, n}}\xrightarrow{\cong}  \rmH^{1}(\FF_{p^{2}}, \rmH^{1}(\overline{X}_{\FF^{\ac}_{p}}, \calL_{k-2}(\calO)(1))_{/I_{f, n}}).
\end{equation}
\end{enumerate}
\end{theorem}

\begin{remark}
We will refer to this theorem as the \emph{unramified arithmetic level raising for the Kuga-Sato varieties}.  It addresses a question raised in the introduction of \cite{CH-2} about the surjectivity of the Abel-Jacobi map restricted to the supersingular locus.  The proof is inspired by lectures of Liang Xiao at the Morning side center. 
\end{remark}

\begin{myproof}{Theorem}{\ref{level-raise-curve}} 
We first proceed to show that $\Phi_{n}$ is surjective. We consider the localized weight spectral sequence for $\rmH^{1}(\interX_{0}(p)\otimes{\QQ^{\ac}_{p}}, \calL_{k-2}(\calO)(1))_{\frakm_{f}}$ and its induced monodromy filtration:
\begin{equation*}
\begin{aligned}
&0\subset^{\rmE^{1,0}_{2,\frakm_{f}}} M_{1}\rmH^{1}(\interX_{0}(p)\otimes{\QQ^{\ac}},\calL_{k-2}(\calO)(1))_{\frakm_{f}}\subset^{\rmE^{0,1}_{2, \frakm_{f}}} M_{0}\rmH^{1}(\interX_{0}(p)\otimes{\QQ^{\ac}},\calL_{k-2}(\calO)(1))_{\frakm_{f}}\\
&\subset^{\rmE^{-1,2}_{2, \frakm_{f}}} M_{-1}\rmH^{1}(\interX_{0}(p)\otimes{\QQ^{\ac}},\calL_{k-2}(\calO)(1))_{\frakm_{f}}.\\
\end{aligned}
\end{equation*}
By \eqref{grad-curve},  we have 
\begin{equation*}
\begin{aligned}
\rmE^{1,0}_{2,\frakm_{f}}&=\ker[\rmH^{0}(\overline{X}_{0}(p)_{\FF^{\ac}_{p}} , a_{1*}\calL_{k-2}(\calO))\xrightarrow{\tau} \rmH^{2}(\overline{X}_{0}(p)_{\FF^{\ac}_{p}}, a_{0*}\calL_{k-2}(\calO)(1))]_{\frakm_{f}}\\
&=\rmH^{0}(X^{\ss}_{\FF^{\ac}_{p}}, \calL_{k-2}(\calO)(1))_{\frakm_{f}};\\
\rmE^{0,1}_{2, \frakm_{f}}&= \rmH^{1}(\overline{X}_{0}(p)_{\FF^{\ac}_{p}}, a_{0*}\calL_{k-2}(\calO)(1))_{\frakm_{f}}\\
&=\rmH^{1}(\overline{X}_{\FF^{\ac}_{p}},\calL_{k-2}(\calO)(1))^{\oplus 2}_{\frakm_{f}};\\
\rmE^{-1,2}_{2, \frakm_{f}}&= \coker[\rmH^{0}(\overline{X}_{0}(p)_{\FF^{\ac}_{p}}, a_{0*}\calL_{k-2}(\calO)(1))\xrightarrow{\rho} \rmH^{0}(\overline{X}_{0}(p)_{\FF^{\ac}_{p}}, a_{1*}\calL_{k-2}(\calO)(1))]_{\frakm_{f}}\\
&=\rmH^{0}(X^{\ss}_{\FF^{\ac}_{p}}, \calL_{k-2}(\calO))_{\frakm_{f}}.\\
\end{aligned}
\end{equation*}
Next we consider the push-forward map 
\begin{equation*}
\rmH^{1}(\interX_{0}(p)\otimes{\Qbar_{p}}, \calL_{k-2}(\calO)(1))_{\frakm_{f}}\xrightarrow{(\pi_{1*}, \pi_{2*})} \rmH^{1}(\interX\otimes{\Qbar_{p}},\calL_{k-2}(\calO)(1))_{\frakm_{f}}^{\oplus 2}.
\end{equation*}
This is surjective by ``Ihara's lemma" Theorem \ref{DT} and Nakayma's lemma. It is well-known that the composite 
\begin{equation*}
\rmE^{1,0}_{2,\frakm_{f}}\hookrightarrow \rmH^{1}(\interX_{0}(p)\otimes{\Qbar_{p}}, \calL_{k-2}(\calO)(1))_{\frakm_{f}}\xrightarrow{(\pi_{1*}, \pi_{2*})} \rmH^{1}(\interX\otimes{\Qbar_{p}}, \calL_{k-2}(\calO)(1))_{\frakm_{f}}^{\oplus 2}
\end{equation*}
is zero. Therefore we obtain the following commutative diagram where we have omitted the coefficient $\calL_{k-2}(\calO)$ in all the terms
\begin{equation*}
\begin{tikzcd}
\rmH^{1}(\interX\otimes{\FF^{\ac}_{p}})(1)_{\frakm_{f}}^{\oplus 2} \arrow[r, "({i_{1*}, i_{2*}})"] \arrow[d, "\cong"] & \rmH^{1}(\interX_{0}(p)\otimes{\Qbar_{p}})(1)_{\frakm_{f}} \arrow[r] \arrow[d, "({\pi_{1*}, \pi_{2*}})"] & \rmH^{0}(X^{\ss}_{\FF^{\ac}_{p}})_{\frakm_{f}}  \arrow[d, "\Phi^{\prime}"] \\
\rmH^{1}(\interX\otimes{\Qbar_{p}})(1)_{\frakm_{f}}^{\oplus 2} \arrow[r, "\nabla"]      & \rmH^{1}(\interX\otimes{\Qbar_{p}})(1)_{\frakm_{f}}^{\oplus2}\arrow[r]                & \coker(\nabla)             
\end{tikzcd}.
\end{equation*}
Here the top row of the diagram is the monodromy filtration of  $\rmH^{1}(\interX\otimes{\Qbar_{p}},\calL_{k-2}(\calO)(1))_{\frakm_{f}}$ which is exact on the right. 
The map $\Phi^{\prime}$ is the one naturally induced by $({\pi_{1*}, \pi_{2*}})$. The map $\nabla$ is by definition given by the composite of $({\pi_{1*}, \pi_{2*}})$ and $({i_{1*}, i_{2*}})$. By \eqref{inter-matrix}, the map $\nabla$ is given by the matrix
\begin{equation*}
\begin{pmatrix}
\mathrm{id} &   \Frob_{p}\\
\Frob_{p} & \mathrm{id}\\ 
\end{pmatrix}
\end{equation*}
since the central element $S_{p}$ has trivial action. It follows then that we have an isomorphism
\begin{equation*}
\coker(\nabla) = \rmH^{1}(\FF_{p^{2}}, \rmH^{1}(\interX\otimes{\QQ^{\ac}_{p}}, \calL_{k-2}(\calO)(1))_{\frakm_{f}}).
\end{equation*}
Since  $({\pi_{1*}, \pi_{2*}})$ is surjective, the map  $\Phi^{\prime}$  is surjective as well. Let $\Phi^{\prime}_{n}$ be the reduction of  $\Phi^{\prime}$  modulo $I_{f,n}$. Therefore we are left to show that $\Phi^{\prime}_{n}$ agree with the map $\Phi_{n}$. To show this, we rely on some results proved in \cite{Illusie2}. More precisely, the natural quotient map induced by the monodromy filtration 
\begin{equation*}
\rmH^{1}(\interX_{0}(p)\otimes{\Qbar_{p}}, \calL_{k-2}(\calO)(1))_{\frakm_{f}}\rightarrow \rmH^{0}(X^{\ss}_{\FF^{\ac}_{p}}, \calL_{k-2}(\calO))_{\frakm_{f}}
\end{equation*} 
factors through $\rmH^{1}(X^{\ord}_{\FF^{\ac}_{p}}, \calL_{k-2}(\calO))$:
\begin{equation*}
\rmH^{1}(\interX_{0}(p)\otimes{\Qbar_{p}}, \calL_{k-2}(\calO)(1))_{\frakm_{f}}\xrightarrow{i^{*}_{1}}\rmH^{1}(X^{\ord}_{\FF^{\ac}_{p}}, \calL_{k-2}(\calO)(1))_{\frakm_{f}} \rightarrow \rmH^{0}(X^{\ss}_{\FF^{\ac}_{p}}, \calL_{k-2}(\calO))_{\frakm_{f}} \rightarrow 0.
\end{equation*}
where the map $\rmH^{1}(X^{\ord}_{\FF^{\ac}_{p}}, \calL_{k-2}(\calO)(1))_{\frakm_{f}} \rightarrow \rmH^{0}(X^{\ss}_{\FF^{\ac}_{p}}, \calL_{k-2}(\calO))_{\frakm_{f}}$ comes from the natural excision exact sequence for  $\rmH^{1}(\interX\otimes{\FF^{\ac}_{p}}, \calO(1))_{\frakm_{f}}$ and the $i^{*}_{1}$ is the pullback of the cohomology of Neaby cycles
\begin{equation*}
\rmH^{1}(\overline{X}_{0}(p)_{\FF^{\ac}_{p}}, R\Psi(\calL_{k-2}(\calO))(1))_{\frakm_{f}}\xrightarrow{i^{*}_{1}}\rmH^{1}(X^{\ord}_{\FF^{\ac}_{p}},R\Psi(\calL_{k-2}(\calO))(1))_{\frakm_{f}}.
\end{equation*}
For the proof of these facts, see \cite[Proposition 1.5]{Illusie2} which extends to non-trivial coefficient. Let $x\in \rmH^{0}(X^{\ss}_{\FF^{\ac}_{p}}, \calL_{k-2}(\calO))_{\frakm_{f}}$ and let $\tilde{x}$ be a preimage of $x$ in $\rmH^{1}(X^{\ord}_{\FF^{\ac}_{p}}, \calL_{k-2}(\calO)(1))_{\frakm_{f}}$. Since $i^{*}_{1}i_{1*}$ is the identity map, we can  take $i_{1*}(\tilde{x})$ as a preimage of $\tilde{x}$ in 
\begin{equation*}
\rmH^{1}(\overline{X}_{0}(p)_{\FF^{\ac}_{p}}, R\Psi(\calL_{k-2}(\calO))(1))_{\frakm_{f}}. 
\end{equation*}
Therefore for $x\in \rmH^{0}(X^{\ss}_{\FF^{\ac}_{p}}, \calL_{k-2}(\calO))$, we have $\Phi^{\prime}(x)=(\pi_{1*}i_{1*}(\tilde{x}), \pi_{2*}i_{2*}(\tilde{x}))=(\tilde{x}, \Frob_{p}(\tilde{x}))$. Since the natural quotient map  
\begin{equation*}
\rmH^{1}(\interX\otimes{\Qbar_{p}}, \calL_{k-2}(\calO)(1))_{\frakm_{f}}\oplus \rmH^{1}(\interX\otimes{\Qbar_{p}}, \calL_{k-2}(\calO)(1))_{\frakm_{f}} \rightarrow \coker(\nabla)
\end{equation*}
is given by sending 
\begin{equation*}
(x, y)\in \rmH^{1}(\interX\otimes{\Qbar_{p}}, \calL_{k-2}(\calO)(1))_{\frakm_{f}}\oplus \rmH^{1}(\interX\otimes{\Qbar_{p}}, \calL_{k-2}(\calO)(1))_{\frakm_{f}}
\end{equation*}
to $(x-\Frob_{p}(y))$ in light of the definition of $\nabla$, we have 
$\Phi^{\prime}(x)=(1-\Frob^{2}_{p})\tilde{x}$. But this is precisely the definition of  $\Phi(x)$. Note we have an isomorphism 
\begin{equation*}
\rmT^{r}_{f, n}\cong  \rmH^{1}(\interX\otimes{\QQ^{\ac}}, \calL_{k-2}(\calO)(1))_{/I_{f, n}}
\end{equation*}
for some positive integer $r$.  By Definition \ref{adm} (3), we have $\rmT_{f, n\mid G_{\QQ_{p}}}\cong \calO_{n}(1)\oplus \calO_{n}$. Then it follows that
\begin{equation*}
\rmH^{1}(\FF_{p^{2}}, \rmH^{1}(\overline{X}_{\FF^{\ac}_{p}}, \calL_{k-2}(\calO)(1))_{/I_{f, n}})\cong\rmH^{1}(\FF_{p^{2}}, \calO_{n}(1)\oplus\calO_{n})^{r}.
\end{equation*}
Therefore $\Frob_{p}$ acts by $\epsilon_{p}$ on $\rmH^{1}(\FF_{p^{2}}, \rmH^{1}(\overline{X}_{\FF^{\ac}_{p}}, \calL_{k-2}(\calO)(1))_{/I_{f, n}})$. By \cite[Proposition 3.8]{Ri-100}, we know that $\Frob_{p}$ acts by $U_{p}$ on $X^{\ss}$ and thus by $p^{\frac{2-k}{2}}U_{p}$ on $\rmH^{0}(X^{\ss}, \calL_{k-2}(\calO))$. From the above discussion, we conclude that we have a surjective morphism
\begin{equation*}
\begin{aligned}
\Phi_{n}: S^{B}_{k}(N^{+}, \calO) &\twoheadrightarrow  \rmH^{1}(\FF_{p^{2}}, \rmH^{1}(\overline{X}_{\FF^{\ac}_{p}}, \calL_{k-2}(\calO)(1))_{/I_{f, n}})\\
& \cong \rmH^{1}(\FF_{p^{2}}, \calO_{n}(1)\oplus\calO_{n})^{r}\\
& \twoheadrightarrow\calO_{n}.\\
\end{aligned}
\end{equation*}
This gives us the desired morphism $\phi^{[p]}_{f, n}: \TT^{[p]}\rightarrow \calO_{n}$ which finishes the proof of $(1)$.

The statement in $(2)$ follows from the main results of \cite[Theorem 2]{DT1} with trivial modification to cover the higher weight case. 

The statement in $(3)$ follows from \cite[Proposition 6.8]{CH-1} and the slight modification in \cite[Proposition 5.9]{Chida} which replaces the ordinary local condition at $l$ by the more general local conditions given by Fontaine-Laffaille theory. It follows then $S^{B}_{k}(N^{+}, \calO)/I^{[p]}_{f, n}$ is of rank $1$ over $\calO_{n}$. Since we have a surjective map $\Phi_{n}: S^{B}_{k}(N^{+}, \calO) \twoheadrightarrow  \rmH^{1}(\FF_{p^{2}}, \rmH^{1}(\overline{X}_{\FF^{\ac}_{p}}, \calL_{k-2}(\calO)(1))_{/I_{f, n}})$, the rank of $\rmH^{1}(\overline{X}_{\FF^{\ac}_{p}}, \calL_{k-2}(\calO)(1))_{/I_{f, n}}$ has to be $2$. 

The statement in $(4)$ follows from the previous discussions. More precisely, by $(3)$, the module $S^{B}_{k}(N^{+}, \calO)_{/I^{[p]}_{f, n}}$ is free of rank one over $\calO_{n}$ and we have a surjective map $\Phi_{n}$. This concludes the proof of this theorem. 
\end{myproof}

\subsection{Ramified level raising on Kuga-Sato varieties} Let $B^{\prime\prime}$ be the indefinite quaternion algebra with discriminant $pp^{\prime}N^{-}$. Let $\calO_{B^{\prime\prime}, N^{+}}$ be an Eichler order of level $N^{+}$ contained in a fixed maximal order $\calO_{B^{\prime\prime}}$. Then we define the Shimura curve $X^{\prime\prime}=X^{B^{\prime\prime}}_{N^{+}, pp^{\prime}N^{-}}$ the same way as we define $X=X^{B^{\prime}}_{N^{+}, N^{-}}$. Then we define an integral model $\interX^{\prime\prime}$ of $X^{\prime\prime}$ over $\ZZ_{p^{\prime}}$. For a $\ZZ_{p^{\prime}}$-scheme $S$, $\interX^{\prime\prime}$ is the set of triples $(A, \iota, C)$ where
\begin{enumerate}
\item $A$ is an $S$-abelian scheme of relative dimension $2$;
\item $\iota: \calO_{B^{\prime}}\hookrightarrow \End_{S}(A)$ is an embedding which is special in the sense of \cite[131-132]{BC-unifor};
\item $C$ is an $\calO_{B^{\prime\prime}}$-stable locally cyclic subgroup of $A[N^{+}]$ of order $N^{+2}$.
\end{enumerate}
This moduli problem is coarsely represented by a projective scheme $\interX^{\prime\prime}$ of relative dimension $1$ over $\ZZ_{p^{\prime}}$. We can similarly rigidify the moduli problem by adding a full level-$d$ structure for a positive integer $d\nmid pp^{\prime}N$ as we did in \eqref{d-full-level}. Then we will write the resulting moduli problem by $\interX^{\prime\prime}_{d}$. The formal completion of $\interX^{\prime\prime}_{d}$ along the special fiber at $p^{\prime}$ admits the Cerednick-Drinfeld uniformization after base change to $\ZZ_{p^{\prime2}}$. The Cerednick-Drinfeld uniformization theorem asserts that the ${\interX}^{\prime\prime\wedge}_{d}$ can be uniformized by the formal scheme $\calM$ which is a disjoint union of the Drinfeld upper-half planes:
\begin{equation}\label{p-unifor}
{\interX}^{\prime\prime\wedge}_{d} \xrightarrow{\sim} G(\QQ)\backslash \calM \times G(\mathbf{A}^{(\infty, p^{\prime})})/K^{p^{\prime}}_{d}. 
\end{equation} 
Here $K$ is the open compact subgroup given by the Eichler order $\calO_{B, N^{+}}$ and $K_{d}$ is given by
\begin{equation*}
K_{d}=\{g=(g_{v})_{v}\in K: g_{v}\equiv \begin{pmatrix}1&0\\0& 1\end{pmatrix}\mod v \text{ for all } v\mid d\}. 
\end{equation*}
Let $\overline{X}^{\prime\prime}_{d}$ be the special fiber of $\interX^{\prime\prime}_{d}$. Then we have the following proposition.
\begin{proposition}\label{curve-red}
We have the following descriptions of the scheme $\overline{X}^{\prime\prime}_{d}$.
\begin{enumerate}
\item The scheme $\overline{X}^{\prime\prime}_{d}$ is a union $\PP^{1}$-bundles over Shimura sets 
\begin{equation*}
\overline{X}^{\prime\prime}_{d}= \PP^{1}(X^{B}_{+})\cup \PP^{1}(X^{B}_{-}). 
\end{equation*}
where both $X^{B}_{+}$ and $X^{B}_{-}$ are isomorphic to the Shimura set $X^{B}_{d}$ as in \eqref{Shi-set}. 
\item The intersection points of the two $\PP^{1}$-bundles $\PP^{1}(X^{B}_{+})$ and $\PP^{1}(X^{B}_{-})$ are given by
\begin{equation*}
\PP^{1}(X^{B}_{+})\cap \PP^{1}(X^{B}_{-})=X^{B}_{d}(p).
\end{equation*}
This also can be identified with the set of singular points on $\overline{X}^{\prime\prime}_{d}$.  A similar statements hold for the curve $\overline{X}^{\prime\prime}$ replacing $X^{B}_{d}$ by $X^{B}$ and $X^{B}_{d}(p)$ by $X^{B}_{0}(p)$. 
\end{enumerate}
\end{proposition}
\begin{proof}
This is well-known. See \cite[Proposition 3.2]{Wang} for an exposition of this result.
\end{proof}
Let $\pi^{\prime\prime}_{k, d}: \calW^{\prime\prime}_{k, d}\rightarrow \interX^{\prime\prime}_{d}$ be the Kuga-Sato variety defined similarly as in \eqref{KS-var}. Then there is a semi-stable model $\tilde{\calW}^{\prime\prime}_{k, d}$ constructed in \cite[Lemma 4]{Saito2} and the action of the idempotent $\epsilon_{k}$ extends naturally to the semi-stable model  $\widetilde{\calW}^{\prime\prime}_{k, d}$. Moreover the first page of the weight spectral sequence converges to $\epsilon_{k}\rmH^{k-1}(\widetilde{\calW}^{\prime\prime}_{k, d}\otimes \QQ_{p^{\prime}}^{\ac}, \Lambda)=\rmH^{1}(\interX^{\prime\prime}_{d}\otimes \QQ_{p^{\prime}}^{\ac}, \calL_{k-2}(\Lambda))$ takes the following form by \cite[page 37]{Saito2}. Let $\TT^{[pp^{\prime}]}$ be the $l$-adic completion of the integral Hecke algebra that acts faithfully on the subspace of $S_{k}(pp^{\prime}N)$ consisting of forms that are new at $pp^{\prime}N^{-}$.
\begin{center}
\begin{tikzpicture}[thick,scale=0.9, every node/.style={scale=0.9}]
  \matrix (m) [matrix of math nodes,
    nodes in empty cells,nodes={minimum width=5ex,
    minimum height=5ex,outer sep=-5pt},
    column sep=1ex,row sep=1ex]{
                &      &     &     & \\
          2     &  \rmH^{0}(\mathfrak{X}^{\prime\prime}_{d}\otimes \FF^{\ac}_{p^{\prime}}, a_{1*}\calL_{k-2}(\Lambda)(-1)) &  \rmH^{2}(\mathfrak{X}^{\prime\prime}\otimes \FF^{\ac}_{p^{\prime}}, a_{0*}\calL_{k-2}(\Lambda))  & & \\
          1     &       & \rmH^{1}(\mathfrak{X}^{\prime\prime}\otimes \FF^{\ac}_{p^{\prime}}, a_{0*}\calL_{k-2}(\Lambda)) &    & \\
          0     &    & \rmH^{0}(\mathfrak{X}^{\prime\prime}\otimes \FF^{\ac}_{p^{\prime}}, a_{0*}\calL_{k-2}(\Lambda)) &  \rmH^{0}(\mathfrak{X}^{\prime\prime}_{d}\otimes \FF^{\ac}_{p^{\prime}}, a_{1*}\calL_{k-2}(\Lambda)) &\\
    \quad\strut &   -1  &  0  &  1  & \strut \\};
\draw[thick] (m-1-1.east) -- (m-5-1.east) ;
\draw[thick] (m-5-1.north) -- (m-5-5.north) ;
\end{tikzpicture}
\end{center}
This means that the weight spectral sequence of the Kuga-Sato variety agree with the weight spectral sequence of the base curve with certain non-trivial coefficient. By further applying the projector $\epsilon_{d}$, we obtain the following first page of the weight spectral sequence converging to 
\begin{equation}\label{Xprimeprime}
\rmH^{1}( \interX^{\prime\prime}\otimes \QQ^{\ac}_{p^{\prime}}, \calL_{k-2}(\Lambda)).
\end{equation}
\begin{center}
\begin{tikzpicture}[thick,scale=0.9, every node/.style={scale=0.9}]
  \matrix (m) [matrix of math nodes,
    nodes in empty cells,nodes={minimum width=5ex,
    minimum height=5ex,outer sep=-5pt},
    column sep=1ex,row sep=1ex]{
                &      &     &     & \\
          2     &  \rmH^{0}(\mathfrak{X}^{\prime\prime}\otimes \FF^{\ac}_{p^{\prime}}, a_{1*}\calL_{k-2}(\Lambda)(-1)) &  \rmH^{2}(\mathfrak{X}^{\prime\prime}\otimes \FF^{\ac}_{p^{\prime}}, a_{0*}\calL_{k-2}(\Lambda))  & & \\
          1     &       & \rmH^{1}(\mathfrak{X}^{\prime\prime}\otimes \FF^{\ac}_{p^{\prime}}, a_{0*}\calL_{k-2}(\Lambda)) &    & \\
          0     &    & \rmH^{0}(\mathfrak{X}^{\prime\prime}\otimes \FF^{\ac}_{p^{\prime}}, a_{0*}\calL_{k-2}(\Lambda)) &  \rmH^{0}(\mathfrak{X}^{\prime\prime}\otimes \FF^{\ac}_{p^{\prime}}, a_{1*}\calL_{k-2}(\Lambda)) &\\
    \quad\strut &   -1  &  0  &  1  & \strut \\};
\draw[thick] (m-1-1.east) -- (m-5-1.east) ;
\draw[thick] (m-5-1.north) -- (m-5-5.north) ;
\end{tikzpicture}
\end{center}
Note that we can make it explicit for the terms in the above spectral sequence
\begin{enumerate}
\item $ \rmH^{0}(\mathfrak{X}^{\prime\prime}\otimes \FF^{\ac}_{p^{\prime}}, a_{1*}\calL_{k-2}(\Lambda))=\rmH^{0}(X^{B}_{0}(p)_{\FF^{\ac}_{p^{\prime}}}, \calL_{k-2}(\Lambda))$;
\item  $\rmH^{1}(\mathfrak{X}^{\prime\prime}\otimes \FF^{\ac}_{p^{\prime}}, a_{0*}\calL_{k-2}(\Lambda))=0$;
\item  $\rmH^{0}(\mathfrak{X}^{\prime\prime}\otimes \FF^{\ac}_{p^{\prime}}, a_{0*}\calL_{k-2}(\Lambda))=\rmH^{0}(X^{B}_{\FF^{\ac}_{p^{\prime}}}, \calL_{k-2}(\Lambda))^{\oplus 2}$.
\end{enumerate}

\begin{theorem}\label{level-raise-curve-ram}
Let $(p, p^{\prime})$ be a pair of $n$-admissible primes for $f$. We assume that the residual Galois representation $\bar{\rho}_{f, \lambda}$ satisfy the assumption $(\mathrm{CR}^{\star})$. Then we have the following statements.
\begin{enumerate}
\item There exists a surjective homomorphism $\phi^{[pp^{\prime}]}_{f, n}: \TT^{[pp^{\prime}]}\rightarrow \calO_{n}$ such that $\phi^{[pp^{\prime}]}_{f, n}$ agrees with $\phi_{f, n}$ at all Hecke operators away from $p$ and sends $(U_{p}, U_{p^{\prime}})$ to $(\epsilon_{p}p^{\frac{k-2}{2}}, \epsilon_{p^{\prime}}p^{\prime\frac{k-2}{2}})$. We will denote by $I^{[pp^{\prime}]}_{f, n}$ the kernel of $\phi^{[pp^{\prime}]}_{f, n}$. 
\item We have an isomorphism of $\calO_{n}$-modules of rank $1$
\begin{equation*}
\Xi_{n}: S^{B}_{k}(N^{+}, \calO){/I^{[p]}_{f, n}}\xrightarrow{\cong}\rmH^{1}_{\sing}(\QQ_{p^{\prime2}}, \rmH^{1}(\interX^{\prime\prime}\otimes{\QQ_{p^{\prime}}^{\ac}}, \calL_{k-2}(\calO)(1)){/I^{[pp^{\prime}]}_{f, n}}).
\end{equation*}
\end{enumerate}
\end{theorem}
\begin{proof}
Following the formulas as proved in \eqref{1-sing}, we have that 
\begin{equation}
\rmH^{1}_{\sing}(\QQ_{p^{\prime2}}, \rmH^{1}(\interX^{\prime\prime}\otimes{\QQ_{p^{\prime}}^{\ac}}, \calL_{k-2}(\calO)(1)))
\end{equation}
is isomorphic to 
\begin{equation}\label{isom-sing}
\coker[\rmH^{0}(X^{\prime\prime}_{\FF^{\ac}_{p^{\prime}}}, a_{0*}\calL_{k-2}(\calO))\xrightarrow{\rho} \rmH^{0}(X^{\prime\prime}_{\FF^{\ac}_{p^{\prime}}}, a_{1*}\calL_{k-2}(\calO))\xrightarrow{\tau}  \rmH^{2}(X^{\prime\prime}_{\FF^{\ac}_{p^{\prime}}}, a_{0*}\calL_{k-2}(\calO(1)))]^{G_{\FF_{p^{\prime2}}}}.
\end{equation}
Here we have 
\begin{equation*}
\rmH^{0}(X^{\prime\prime}_{\FF^{\ac}_{p^{\prime}}}, a_{0*}\calL_{k-2}(\calO))=\rmH^{0}(X^{B}_{\FF^{\ac}_{p^{\prime}}}, \calL_{k-2}(\calO))^{\oplus2}
\end{equation*}
and we can identify it with the space
$S^{B}_{k}(N^{+},\calO)^{\oplus 2}$. Similarly under the Poincare duality, we can also identify  
\begin{equation*}
\rmH^{2}(X^{\prime\prime}_{\FF^{\ac}_{p^{\prime}}}, a_{0*}\calL_{k-2}(\calO)(1))=\rmH^{2}(\PP^{1}(X^{B}_{\FF^{\ac}_{p^{\prime}}}), \calL_{k-2}(\calO)(1))^{\oplus2}
\end{equation*}
with $S^{B}_{k}(N^{+}, \calO)^{\oplus 2}$. The space  $\rmH^{0}(X^{\prime\prime}_{\FF^{\ac}_{p^{\prime}}}, a_{1*}\calL_{k-2}(\calO))$ can be identified with $S^{B}_{k}(pN^{+},\calO)$. Under these identifications, the composition 
\begin{equation*}
\rmH^{0}(X^{\prime\prime}_{\FF^{\ac}_{p^{\prime}}}, a_{0*}\calL_{k-2}(\calO))\xrightarrow{\rho} \rmH^{0}(X^{\prime\prime}_{\FF^{\ac}_{p^{\prime}}}, a_{1*}\calL_{k-2}(\calO))\xrightarrow{\tau}  \rmH^{2}(X^{\prime\prime}_{\FF^{\ac}_{p^{\prime}}}, a_{0*}\calL_{k-2}(\calO(1)))
\end{equation*}
is given by the \emph{intersection matrix}
\begin{equation*}
\begin{pmatrix}
-p^{\prime\frac{k-2}{2}}(p^{\prime}+1) &T_{p^{\prime}}\\
T_{p^{\prime}} &-p^{\prime\frac{k-2}{2}}(p^{\prime}+1)\\
\end{pmatrix}
\end{equation*}
which we will also denote by $\nabla$. Since $p, p^{\prime}$ are $n$-admissible for $f$, the module
\begin{equation*}
\coker[S^{B}_{k}(N^{+},\calO)^{\oplus 2}_{/I^{[p]}_{f, n}}\xrightarrow{\nabla}S^{B}_{k}(N^{+},\calO)^{\oplus 2}_{/I^{[p]}_{f, n})}]
\end{equation*}
is of rank one over $\calO_{n}$ and is isomorphic to $S^{B}_{k}(N^{+},\calO)_{/I^{[p]}_{f, n}}$. Note here the isomorphism between 
\begin{equation*}
\coker[S^{B}_{k}(N^{+},\calO)^{\oplus 2}_{/I^{[p]}_{f, n}}\xrightarrow{\nabla}S^{B}_{k}(N^{+},\calO)^{\oplus 2}_{/I^{[p]}_{f, n})}]
\end{equation*}
and $S^{B}_{k}(N^{+},\calO)^{\oplus 2}_{/I^{[p]}_{f, n}}$ is induced by the map $(x, y)\mapsto \frac{1}{2}(x+\epsilon_{p^{\prime}}y)$ for $(x, y)\in S^{B}_{2}(N^{+},\calO)^{\oplus 2}_{/I^{[p]}_{f, n}}$.
By \cite[Theorem 5.8]{BD-Main} and \cite[\S3.5]{CH-2} adapted to the higher weight case, the natural $U_{p}$-action on 
\begin{equation*}
\rmH^{2}(X^{\prime\prime}_{\FF^{\ac}_{p^{\prime}}}, a_{0*}\calL_{k-2}(\calO)(1))\cong S^{B}_{2}(N^{+},\calO)^{\oplus 2}
\end{equation*}
is given by $(x, y)\mapsto (-p^{\prime\frac{k}{2}}y, p^{\prime\frac{k-2}{2}}x+T_{p^{\prime}}y)$.
We consider the automorphism 
\begin{equation*}
\delta: S^{B}_{2}(N^{+},\calO)^{\oplus 2} \rightarrow S^{B}_{2}(N^{+},\calO)^{\oplus 2}
\end{equation*}
given by 
$(x, y)\mapsto ( p^{\prime\frac{k-2}{2}}x+T_{p^{\prime}}y, p^{\prime\frac{k-2}{2}}y)$.
Then a quick calculation gives us that $\nabla\circ \delta=p^{\prime-(k-2)}U^{2}_{p^{\prime}}-1$.  This means that the quotient 
\begin{equation*}
\frac{S^{B}_{k}(N^{+},\calO)^{\oplus 2}}{(I^{[p]}_{f, n}, U^{2}_{p^{\prime}}-p^{\prime k-2})} \cong \coker[S^{B}_{k}(N^{+},\calO)^{\oplus 2}_{/I^{[p]}_{f, n}}\xrightarrow{\nabla}S^{B}_{k}(N^{+},\calO)^{\oplus 2}_{/I^{[p]}_{f, n})}]\end{equation*}
is of rank $1$. Since $p^{\prime}$ is $n$-admissible for $f$, we see immediately $U_{p^{\prime}}+\epsilon_{p^{\prime}}p^{\prime\frac{k-2}{2}}$ is invertible on $S^{B}_{k}(N^{+},\calO)^{\oplus 2}_{/I^{[p^{\prime}]}_{f, n}}$. Therefore we have
\begin{equation*}
\frac{S^{B}_{2}(N^{+},\calO)^{\oplus 2}}{(I^{[p]}_{f, n},U^{2}_{p^{\prime}}-p^{\prime k-2})} \cong \frac{S^{B}_{2}(N^{+},\calO)^{\oplus 2}}{(I^{[p]}_{f, n},U_{p^{\prime}}-\epsilon_{p^{\prime}}p^{\prime \frac{k-2}{2}})}
\end{equation*}
and the latter quotient is of rank $1$ over $\calO_{n}$. Then the action of $\TT^{[pp^{\prime}]}$ on this rank $1$ quotient gives the desired morphism 
$\phi^{[pp^{\prime}]}_{f, n}: \TT^{[pp^{\prime}]}\rightarrow \calO_{n}$.  This finishes the proof of the part $(1)$. Part $(2)$ follows directly from part $(1)$ using the isomorphism \eqref{isom-sing}
\end{proof}
\begin{remark}
We remark that a similar ramified arithmetic level raising theorem is first proved by Chida in \cite[Theorem 5.11]{Chida}. 
\end{remark}

\section{Heegner cycles over Shimura curves}

\subsection{Heegner cycles over Shimura curves} Let $K$ be an imaginary quadratic field with discriminant $-D_{K}<0$ and set $\delta_{K}=\sqrt{-D_{K}}$. Let $z\mapsto \bar{z}$ be the complex conjugate action on $K$.  We define $\boldsymbol{\theta}$ by  
\begin{equation*}
\boldsymbol{\theta}=\frac{D^{\prime}+\delta_{K}}{2}, \hphantom{a}D^{\prime}= 
\begin{cases}
D_{K}\hphantom{a}  &\text{if }2\nmid D_{K}\\ 
D_{K}/2\hphantom{a}& \text{if }2\mid D_{K}.
\end{cases}
\end{equation*}
We always fix a positive integer $N$ such that $N=N^{+}N^{-}$ with $N^{+}$ consists of prime factors that are split in $K$ while $N^{-}$ consists of prime factors that are inert in $K$. We will assume the following generalized Heegner hypothesis
\begin{equation*}\tag{Heeg}
\text{\emph{$N^{-}$ is square free and consists of even number of prime factors that are inert in $K$}}.
\end{equation*}
Let $B^{\prime}$ be the indefinite quaternion algebra of discriminant $N^{-}$. We can regard $K$ as a sub-algebra of $B^{\prime}$ via an embedding $\iota: K\hookrightarrow B^{\prime}$.  Let $m$ be a positive integer such that $(m, Nl)=1$. We will chose an element $J$ such that 
\begin{equation}\label{J}
B^{\prime}=K\oplus K\cdot J 
\end{equation}
and satisfies the following properties
\begin{enumerate}
\item $J^{2}=\beta\in\QQ^{\times}$ with $\beta<0$ and $Jt=\bar{t}J$ for all $t\in K$;
\item $\beta\in(\ZZ^{\times}_{q})^{2}$ for all $q\mid N^{+}$ and $\beta\in\ZZ^{\times}_{q}$ for $q\mid D_{K}$.
\end{enumerate}

We define $\varsigma_{q}\in G^{\prime}(\QQ_{q})$ as follows
\begin{equation}
\varsigma_{q}=
\begin{cases}
1 &\text{if $q\nmid mN^{+}$};\\

\delta^{-1}\begin{pmatrix}\boldsymbol{\theta} & \bar{\boldsymbol{\theta}}\\ 1& 1\\ \end{pmatrix}& \text{if $q=\mathfrak{q}\bar{\mathfrak{q}}$ is split with $\mathfrak{q}\mid \mathfrak{N}^{+}$ }\\

\begin{pmatrix}q^{n} & 0\\ 0& 1\\ \end{pmatrix}& \text{if $q\mid m$ and $q$ is inert in $K$ with $n=\ord_{q}(m)$ }\\

\begin{pmatrix}1 & q^{-n}\\ 0& 1\\ \end{pmatrix}& \text{if $q\mid m$ and $q$ is split in $K$ with $n=\ord_{q}(m)$ }\\
\end{cases}
\end{equation}

We define the \emph{Atkin-Lehner involution} at $q$ to be 
\begin{equation}
\tau_{q}=
\begin{cases}
\begin{pmatrix}0& 1\\ -N^{+}& 1\\ \end{pmatrix} &\text{for $q\mid N^{+}$};\\
1 &\text{for $q\nmid N^{+}$}.
\end{cases}
\end{equation}
Then we put $\tau^{N^{+}}=\prod_{q}\tau_{q}$ as an element in $G^{\prime}(\Adel)$. Next we define the set of CM points on $X$. We define $z^{\prime}$ to be the fixed point in $\calH^{\pm}$ by  $\iota_{\infty}(K)\subset \GL_{2}(\mathbf{R})$. We define the set of CM points on the Shimura curve $X$ by
\begin{equation*}
\mathrm{CM}_{K}(X)=\{[z^{\prime}, b^{\prime}]_{\CC}: b^{\prime}\in G^{\prime}(\Adel^{(\infty)})\}.
\end{equation*}
Let $\mathrm{rec}_{K}: \widehat{K}^{\times}\rightarrow \Gal(K^{\mathrm{ab}}/K)$ be the geometrically normalized reciprocity law. The Shimura reciprocity law says that
\begin{equation*}
\mathrm{rec}_{K}(a)[z^{\prime}, b^{\prime}]=[z^{\prime}, \iota(a)b^{\prime}].
\end{equation*}
 Let $m$ be a positive integer that is prime to $N$. Let $\calO_{K, m}=\ZZ+m\calO_{K}$ be the the order of $K$ with conductor $m$.  Let $\mathcal{G}_{m}=K^{\times}\backslash \widehat{K}^{\times}/ \widehat{\calO}^{\times}_{K, m}$ be the Galois group of the ring class field $K_{m}$ of conductor $m$ over $K$. Let $a\in \widehat{K}^{\times}$ and $p$ be a prime inert in $K$, then we define the \emph{CM point of level $m$} on $X$ by
\begin{equation}\label{Heeg-pt}
{P}_{m}(a)=[z^{\prime}, a^{(p)}\varsigma\tau^{N^{+}}]_{\CC}\in X(\CC).
\end{equation}
We set $P_{m}:={P}_{m}(1)$ and call this the \emph{Heegner point of level $m$}. This by definition gives a point in $X(K_{m})$ which has the following moduli interpretation. The point ${P}_{m}$ corresponds to a triple $(A_{m}, \iota_{m}, C_{m})$ such that $\End(A_{m}, \iota_{m})$ is isomorphic to $\calO_{K, m}$.  We let $(\tilde{P}_{m}(a), \tilde{P}_{m})$ be an arbitrary lift of $(P_{m}(a), P_{m})$ in $X_{d}(K_{m})$.

We define Heegner cycles over $X_{d}$ and $X$ following  \cite{Nekovar-Heeg} in the classical modular curve case and \cite{IS}, \cite{Chida} in the Shimura curve case.  Consider the point $\tilde{P}_{m}=({A}_{m}, {\iota}_{m}, {C}_{m}, \nu_{m})$ and the Neron-Severi group $\mathrm{NS}({A}_{m})$ of ${A}_{m}$. There is a natural action of $B^{\prime\times}$ on $\mathrm{NS}({A}_{m})_{\QQ}$ given by $\calL\cdot b=\iota_{m}(b)^{*}\calL$. Since ${A}_{m}$ admits an action of $\calO_{B^{\prime}}\otimes \calO_{K, m}\cong \rmM_{2}(\calO_{K, m})$, it is clear that ${A}_{m}$ is isomorphic to a product $E_{m}\times E_{m}$ with $E_{m}$ an elliptic curve with CM by $\calO_{K, m}$. Let $\Gamma_{m}$ be the graph of $m\sqrt{-D_{K}}$ in ${A}_{m}=E_{m}\times E_{m}$. Then we define $Z_{m}$ to be the image of the divisor given by $[\Gamma_{m}]-[E_{m}\times 0]-mD_{K}[0\times E_{m}]$ in $\mathrm{NS}(A_{m})$. It lies in the rank one submodule of $\mathrm{NS}({A}_{m})$ generated by $\langle[0\times E_{m}], [E_{m}\times 0], \Delta_{m}\rangle$ where $\Delta_{m}$ is the diagonal.  Let $y_{m}\in \mathrm{NS}({A}_{m})\otimes\ZZ_{l}$ be the class representing $m^{-1}Z_{m}$. This is the unique class up to sign satisfying 
\begin{enumerate}
\item $\iota_{m}(b)^{*}(y_{m})=\Nrd(b)y_{m}$ for any $b\in B^{\prime}$;
\item The self-intersection number of $y_{m}$ is $2D_{K}$.
\end{enumerate}
Taking the $\frac{k-2}{2}$-th exterior product of the element $\epsilon_{2}y_{m}\in \epsilon_{2}\mathrm{NS}({A}_{m})\otimes\ZZ_{l}\cong \epsilon_{2}\CH^{1}({A}_{m})\otimes\ZZ_{l}$, we obtain an element $\epsilon_{k}y^{\frac{k-2}{2}}_{m}\in \epsilon_{k} \CH^{\frac{k-2}{2}}({A}^{\frac{k-2}{2}}_{m})\otimes\ZZ_{l}$. Denote by the embedding $j_{k, d}: {A}^{\frac{k-2}{2}}_{m}\hookrightarrow \calW_{k, d}$ which if of codimension $1$. We have the push-forward map 
\begin{equation}
\begin{aligned}
\epsilon_{k}\CH^{\frac{k-2}{2}}({A}^{\frac{k-2}{2}}_{m})\otimes \ZZ_{l}\xrightarrow{j_{k, d*}}\epsilon_{k}\CH^{\frac{k}{2}}(\calW_{k, d})\otimes\ZZ_{l}.
\end{aligned}
\end{equation}
Then we define the \emph{Heegner cycle} $Y_{m, k}$ in $\epsilon_{k}\CH^{\frac{k}{2}}(\calW_{k, d}\otimes K_{m})\otimes\ZZ_{l}$ by 
\begin{equation}
Y_{m, k}:=j_{k, d*}(\epsilon_{k}y^{\frac{k-2}{2}}_{m}).
\end{equation}
Next we consider the Abel-Jacobi map for $X_{d}$ and the local system $\calL_{k-2}$:
\begin{equation*}
\begin{aligned}
\mathrm{AJ}_{k, d}: \epsilon_{k}\CH^{\frac{k}{2}}(\calW_{k, d}\otimes K_{m})\otimes\ZZ_{l}&\rightarrow  \rmH^{1}(K_{m}, \epsilon_{k}\rmH^{k-1}(\calW_{k, d}\otimes \QQ^{\ac}, \ZZ_{l}(\frac{k}{2})))\\
&\cong \rmH^{1}(K_{m}, \rmH^{1}(\interX_{d}\otimes\QQ^{\ac}, \calL_{k-2}(\ZZ_{l})(1))).\\
\end{aligned}
\end{equation*}
We can further apply the projector $\epsilon_{d}$ and it induces the following  Abel-Jacobi map for $X$ and the local system $\calL_{k-2}$:
\begin{equation*}
\mathrm{AJ}_{k}: \epsilon_{d}\epsilon_{k}\CH^{\frac{k}{2}}(\calW_{k, d}\otimes K_{m})\otimes\ZZ_{l} \rightarrow  \rmH^{1}(K_{m}, \rmH^{1}(\interX\otimes\QQ^{\ac}, \calL_{k-2}(\ZZ_{l})(1))).\\
\end{equation*}
Finally, we compose this map with the canonical map from $\rmH^{1}(\interX\otimes\QQ^{\ac}, \calL_{k-2}(\ZZ_{l})(1))$ to 
\begin{equation*}
\rmH^{1}(\interX\otimes\QQ^{\ac}, \calL_{k-2}(\ZZ_{l})(1))_{\frakm_{f}}\otimes_{\TT_{\frakm_{f}}}\calO\cong \mathrm{T}_{f,\lambda}
\end{equation*}
 where the tensor product is induced by $\phi_{f}: \TT_{\frakm_{f}}\rightarrow \calO$. We therefore have the following Abel-Jacobi map for the representation $\mathrm{T}_{f,\lambda}$:
\begin{equation}\label{AJ-T}
\mathrm{AJ}_{f, k}: \epsilon_{d}\epsilon_{k}\CH^{\frac{k}{2}}(\calW_{k, d}\otimes K_{m})\otimes\ZZ_{l} \rightarrow  \rmH^{1}(K_{m}, \mathrm{T}_{f, \lambda}).\\
\end{equation}
We will define the \emph{level $m$ Heegner cycle class } by
\begin{equation*}
\kappa(m):=\mathrm{AJ}_{f, k}(\epsilon_{d}\epsilon_{k}Y_{m, k})\in  \rmH^{1}(K_{m}, \mathrm{T}_{f, \lambda}).
\end{equation*} 
We will refer to the following class simply as the  \emph{Heegner cycle class}:
\begin{equation}\label{Heegner-cycle}
\kappa:=\mathrm{Cor}_{K_{1}/K}\kappa \in  \rmH^{1}(K, \mathrm{T}_{f, \lambda}).
\end{equation}
For $n\geq 1$, we define similarly the Abel-Jacobi map for the representation $\mathrm{T}_{f,{n}}$:
\begin{equation}
\mathrm{AJ}_{k, n}: \epsilon_{d}\epsilon_{k}\CH^{\frac{k}{2}}(\calW_{k, d}\otimes K_{m})\otimes\ZZ_{l} \rightarrow  \rmH^{1}(K_{m}, \mathrm{T}_{f, n}).\\
\end{equation}
Reducing the classes $\kappa(m)$ and $\kappa$ modulo $\varpi^{n}$, we define 
\begin{equation}
\begin{aligned}
&\kappa_{n}(m)\in\rmH^{1}(K_{m}, \mathrm{T}_{f, n});\\
&\hphantom{aa}\kappa_{n}\in \rmH^{1}(K, \mathrm{T}_{f, n}).\\
\end{aligned}
\end{equation} 

\subsection{Theta element and special value formula} Let $p$ be a prime away from $N$ and consider the definite quaternion algebra $B$ over $\QQ$ with discriminant $pN^{-}$. We denote by $G$ the algebraic group over $\QQ$ given by $B^{\times}$. We will choose the element $J^{\prime}$ as in \eqref{J} such that $B=K\oplus K\cdot J^{\prime}$. For each $a\in \widehat{K}$, we define the \emph{Gross points} of conductor $m$ associated to $K$ by
\begin{equation}
x_{m}(a):=a\cdot \varsigma\in G(\Adel).
\end{equation}

Recall we have the fixed embedding $\iota_{l}: \QQ^{\ac}\hookrightarrow \CC_{l}$ and it induces the place $\mathfrak{l}$ of $K$ and the place $\lambda$ of $\QQ^{\ac}$. We define an embedding
\begin{equation}
i_{K}:  B\rightarrow \rmM_{2}(K), \hphantom{a} a+bJ^{\prime}\mapsto  \begin{pmatrix}a & b\beta \\ \bar{b}& \bar{a}\\ \end{pmatrix}, \hphantom{a}a, b\in K.
\end{equation}
and let $i_{\CC}:= \iota_{\infty}\circ i_{K}$ and $i_{K_{\mathfrak{l}}}=\iota_{l}\circ i_{K}$be the composition.  Let $\rho_{k,\infty}$ be the representation
\begin{equation}
\rho_{k,\infty}: G(\RR)\xrightarrow{i_{\CC}} \GL_{2}(\CC)\rightarrow \mathrm{Aut}_{\CC}L_{k-2}(\CC).
\end{equation}
Then $\CC\cdot \mathbf{v}_{r}$ is the line on which $\rho_{k, \infty}(t)$ acts by $(\bar{t}/{t})^{r}$ for $t\in (K\otimes \CC)^{\times}$. For a $K$-algebra $A$ we define the space $\mathbf{S}^{B}_{k}(U, A)$ of modular forms on $B$ of weight $k$ and level $U$ to be:
\begin{equation*}
\{h:G(\mathbf{A}^{(\infty)})\rightarrow L_{k-2}(A): h(\alpha g u)=\rho_{k, \infty}(\alpha)h(g) \text{ for $\alpha\in G(\QQ)$ and $u\in U$}\}.
\end{equation*}
Let $\mathbf{S}^{B}_{k}(\CC)=\varinjlim_{U}\mathbf{S}^{B}_{k}(U, \CC)$ and $\calA(G)$ be the space of automorphic forms on $G(\Adel)$. We define a morphism 
\begin{equation*}
\Psi: L_{k-2}(\CC)\otimes \mathbf{S}^{B}_{k}(\CC)\rightarrow \calA(G)  
\end{equation*}
by the following recipe
\begin{equation*}
\Psi(\mathbf{v}\otimes f)(g):=\langle\rho_{k,\infty}(g_{\infty})\mathbf{v}, f(g^{\infty})\rangle_{k-2}
\end{equation*}
for $\mathbf{v}\in L_{k-2}(\CC)$. Let $\pi$ be the automorphic representation of $\GL_{2}(\Adel)$ corresponding to $f^{[p]}$ and $\pi^{\prime}$ be the automorphic representation of $G(\Adel)$ that corresponds to $\pi$ via the Jacquet-Langlands correspondence. Let $f^{[p]}_{\pi^{\prime}}$ be a generator of $\mathbf{S}^{B}(N^{+}, \CC)[\pi^{\prime}_{f}]$. We define an automorphic form in $\calA(G)$ by
\begin{equation}
\varphi^{[p]}_{\pi^{\prime}}:=\Psi(\mathbf{v}^{*}_{0}\otimes f^{[p]}_{\pi^{\prime}}) \text{ for } \mathbf{v}^{*}_{0}=D^{\frac{k-2}{2}}_{K}\mathbf{v}_{0}.
\end{equation}
Let $\rho_{k,l}$ to be the representation defined by
\begin{equation}
\rho_{k, l}: G(\QQ)\xrightarrow{i_{K_{\mathfrak{l}}}} \GL_{2}(\CC_{l})\rightarrow \mathrm{Aut}_{\CC_{l}}L_{k-2}(\CC_{l}).
\end{equation}
It is easy to check that $\rho_{k}$ and $\rho_{k, l}$ are compatible in the sense that
\begin{equation}
\rho_{k, l}(g)=\rho_{k}(\gamma_{\frakl}i_{l}(g)\gamma^{-1}_{\frakl}), \text{ where } \gamma_{\frakl}:=\begin{pmatrix}\sqrt{\beta}& -\sqrt{\beta}\bar{\boldsymbol{\theta}}\\-1& \boldsymbol{\theta}\\ \end{pmatrix}\in \GL_{2}(K_{\frakl}).
\end{equation}
If $l$ is invertible in $A$, then we have in fact an isomorphism
\begin{equation}
\mathbf{S}^{B}_{k}(N^{+}, A)\xrightarrow{\cong} {S}^{B}_{k}(N^{+}, A), \hphantom{a} h\mapsto \widehat{h}(g):=\rho_{k}(\gamma_{\frakl}^{-1})\rho_{k, l}(g^{-1}_{l})h(g)
\end{equation}
and we say $\widehat{h}$ is an \emph{$l$-adic avatar} of $h$. We will say $f^{[p]}_{\pi^{\prime}}$ is $l$-adically normalized if $\widehat{f}^{[p]}_{\pi^{\prime}}$ is a generator of the rank one module $S^{B}_{k}(N, \calO)[\pi^{\prime}_{f}]:=S^{B}_{k}(N, \calO)\cap S^{B}_{k}(N, \CC_{l})[\pi^{\prime}_{f}]$.  We can now define the theta element associated to $f$ and $K$. Let $f^{[p]}_{\pi^{\prime}}$ be $\lambda$-adically normalized, we define the theta element $\Theta_{m}(f^{[p]}_{\pi^{\prime}})\in \calO[\mathcal{G}_{m}]$ by
\begin{equation}\label{Theta}
\Theta_{m}(f^{[p]}_{\pi^{\prime}})=\sum_{\sigma\in \mathcal{G}_{m}}\varphi^{[p]}_{\pi^{\prime}}(\sigma\cdot x_{m}(1))[\sigma]. 
\end{equation}
We will denote the theta element simply by  $\Theta(f^{[p]}_{\pi^{\prime}})$ if $m=1$. The following theorem relates the  central critical value of the $L$-function of $f^{[p]}$ over $K$ twisted by a ring class character $\chi$ of $\mathcal{G}_{m}$ to the theta element above. 
\begin{theorem}[Chida-Hsieh, Hung] \label{special-value}
Let $\chi$ be character of $\mathcal{G}_{m}$ and $N^{+}=\gothN^{+}\cdot\overline{\gothN^{+}}$. Then we have 
\begin{equation*}
\chi(\Theta_{m}(f^{[p]}_{\pi^{\prime}})^{2})=\Gamma(k/2)^{2}\cdot \frac{L(f^{[p]}/K, \chi, k/2)}{\Omega_{\pi^{\prime}, N^{-}}}\cdot (-1)^{m}\cdot D^{k-1}_{K}\cdot \frac{\vert\calO^{\times}_{K}\vert^{2}}{8}\cdot \sqrt{-D_{K}}^{-1}\cdot\chi(\gothN^{+})
\end{equation*}
where $\Omega_{f^{[p]}, pN^{-}}$ is the $l$-adically normalized period for $f^{[p]}$ given by
\begin{equation*}
\Omega_{f^{[p]}, pN^{-}}:=\frac{4^{k-1}\pi^{k}\vert\vert f^{[p]} \vert\vert_{\Gamma_{0}(pN)}}{\langle f^{[p]}_{\pi^{\prime}}, f^{[p]}_{\pi^{\prime}}\rangle_{B}}.
\end{equation*} 
\end{theorem}
\begin{proof}
This follows from the main result of \cite{Hung} generalizing \cite{CH-1} to ramified characters. 
\end{proof}
Note here the period $\Omega_{f^{[p]}, pN^{-}}$ is not the canonical period $\Omega^{\mathrm{can}}_{f^{[p]}}$ of Hida defined by  
\begin{equation*}
 \Omega^{\mathrm{can}}_{f^{[p]}}:=\frac{4^{k-1}\pi^{k}\vert\vert f^{[p]} \vert\vert_{\Gamma_{0}(pN)}}{\eta_{f^{[p]}}(pN)}
 \end{equation*}
 with $\eta_{f^{[p]}}(pN)$ the congruence number of $f^{[p]}$ in $S_{k}(pN)$. 
 We record the following result comparing these two periods which we will use in a later occasion. Let $q$ be a prime and recall the local Tamagawa ideal $\mathrm{Tam}_{q}(\rmT_{f, \lambda})$ at $q$ is defined by
\begin{equation*}
\mathrm{Tam}_{q}(\rmT_{f, \lambda})=\mathrm{Fitt}_{\calO}(\rmH^{1}(K^{\mathrm{ur}}_{q}, \rmT_{f, \lambda})_{\mathrm{tor}})
\end{equation*}
and the local Tamagawa exponent at $q$ is defined by the number $t_{q}(f)$ such that
\begin{equation*}
\mathrm{Tam}_{q}(\rmT_{f, \lambda})=(\varpi^{t_{q}(f)}). 
\end{equation*}
\begin{proposition}[Kim-Ota] The following equation holds.
\begin{equation*}
v_{\varpi}(\frac{\Omega_{f^{[p]}, pN^{-}}}{\Omega^{\mathrm{can}}_{f^{[p]}}})=\sum_{q\mid pN^{-}} t_{q}(f).
\end{equation*}
\end{proposition}
\begin{proof}
This follows from \cite[Corollary 5.8]{OK} generalizing the work of Pollack-Weston \cite{PW} in weight $2$. 
\end{proof}

\subsection{Explicit reciprocity laws for Heegner cycles} Recall we have the modular form $f\in S^{\new}_{k}(N)$ with $N=N^{+}N^{-}$ such that $N^{-}$ is square free with even number of prime divisors.  Let $n\geq 1$, we consider the Abel-Jacobi map for $\rmT_{f, n}$
\begin{equation*}
\mathrm{AJ}_{k, n}: \epsilon_{d}\epsilon_{k}\CH^{\frac{k}{2}}(\calW_{k, d}\otimes K_{m})\otimes\ZZ_{l} \rightarrow  \rmH^{1}(K_{m}, \mathrm{T}_{f, n}).
\end{equation*}
We have the Heegner cycle class  $\epsilon_{d}Y_{m, k}\in  \epsilon_{d}\epsilon_{k}\CH^{\frac{k}{2}}(\calW_{k, d}\otimes K_{m})\otimes\ZZ_{l}$ with $Y_{m, k}=\epsilon_{k}y^{\frac{k-2}{2}}_{m}$ for an element  $y_{m}\in \mathrm{NS}(A_{m})$ satisfying 
\begin{enumerate}
\item $\iota_{m}(b)^{*}(y_{m})=\Nrd(b)y_{m}$ for any $b\in B^{\prime}$;
\item The self-intersection number of $y_{m}$ is $2D_{K}$.
\end{enumerate}
Here $A_{m}$ is given by the Heegner point $P_{m}=(A_{m}, \iota_{m}, C_{m})$ on $X$. 
Let $p$ be an $n$-admissible prime for $f$. We consider the following composite map 
\begin{equation*}
\epsilon_{d}\epsilon_{k}\CH^{\frac{k}{2}}(\calW_{k, d}\otimes K_{m}) \xrightarrow{\mathrm{AJ}_{k, n}} \rmH^{1}(K_{m}, \rmT_{f, n}) \xrightarrow{\mathrm{loc}_{p}}\rmH^{1}(K_{m, p}, \rmT_{f, n}). 
\end{equation*}
The image of  $\epsilon_{d}Y_{m, k}\in  \epsilon_{d}\epsilon_{k}\CH^{\frac{k}{2}}(\calW_{k, d}\otimes K_{m})\otimes\ZZ_{l}$ under the above map is by definition given by $\mathrm{loc}_{p}(\kappa_{n}(m))$ and it lands in $\rmH^{1}_{\mathrm{fin}}(K_{m, p}, \rmT_{f, n})$ as $\calW_{k, d}$ has good reduction at $p$. Note that there is an isomorphism
\begin{equation}
\rmH^{1}_{\mathrm{fin}}(K_{m, p}, \rmT_{f, n})\cong \rmH^{1}_{\mathrm{fin}}(K_{p}, \rmT_{f, n})\otimes \calO_{n}[\mathcal{G}_{m}].
\end{equation}
by \cite[Lemma 2.4 and 2.5]{BD-Main} and \cite[Lemma 1.4]{CH-2}. Therefore Theorem \ref{level-raise-curve} implies the following isomorphism 
\begin{equation}
\begin{aligned}
\Phi_{n}:  S^{B}_{k}(N^{+}, \calO)_{/I^{[p]}_{f, n}}\otimes \calO_{n}[\mathcal{G}_{m}] &\xrightarrow{\cong}  \rmH^{1}(\FF_{p^{2}}, \rmH^{1}(\overline{X}\otimes{\FF^{\ac}_{p}}, \calL_{k-2}(\calO)(1))_{/I_{f, n}})\otimes \calO_{n}[\mathcal{G}_{m}]\\
&\xrightarrow{\cong}  \rmH^{1}_{\mathrm{fin}}(K_{m, p}, \rmT_{f, n}).\\
\end{aligned}
\end{equation}
It follows then that $\mathrm{loc}_{p} (\kappa_{n}(m))$ can be regarded as an element in  $S^{B}_{k}(N^{+}, \calO)_{/I^{[p]}_{f, n}}\otimes \calO_{n}[\mathcal{G}_{m}]$. Recall $\tilde{P}_{m}=({A}_{m}, {\iota}_{m}, {C}_{m}, \nu_{m})\in X_{d}(K_{m})$ is a lift of the Heegner point $P_{m}=({A}_{m}, {\iota}_{m}, {C}_{m})\in X(K_{m})$ and $A_{m}\cong E_{m}\times E_{m}$ for a CM elliptic curve $E_{m}$. We have the following commutative diagram
\begin{equation*}
\begin{tikzcd}
\epsilon_{d}\epsilon_{k}\mathrm{CH}^{\frac{k-2}{2}}(A^{\frac{k-2}{2}}_{m}\otimes K_{m})\otimes\ZZ_{l} \arrow[r, "cl"] \arrow[d, "{j_{k, d*}}"] &\epsilon_{d}\epsilon_{k} \rmH^{k-2}(A^{\frac{k-2}{2}}_{m}\otimes K_{m}, \ZZ_{l}(\frac{k-2}{2})).          
\arrow[d, "{j_{k, d*}}"] \\
\epsilon_{d}\epsilon_{k}\mathrm{CH}^{\frac{k}{2}}(\calW_{k-2, d}\otimes K_{m})\otimes\ZZ_{l} \arrow[r, "cl"]                & \epsilon_{d}\epsilon_{k}\rmH^{k}(\calW_{k, d}\otimes K_{m}, \ZZ_{l}(\frac{k}{2})).          
\end{tikzcd}
\end{equation*}
We also have the following isomorphisms
\begin{equation*}
\begin{aligned}
\epsilon_{d}\epsilon_{k}\rmH^{k-2}(A^{\frac{k-2}{2}}_{m}\otimes K_{m}, \ZZ_{l}(\frac{k-2}{2}))&=\epsilon_{d}\epsilon_{k}\rmH^{k-2}(E^{k-2}_{m}\otimes K_{m}, \ZZ_{l}(\frac{k-2}{2}))\\
&=\mathrm{Sym}^{k-2} \rmH^{1}(E_{m}\otimes K_{m}, \ZZ_{l}(\frac{k-2}{2}))\\
&\cong L_{k-2}(\ZZ_{l})\\
\end{aligned}
\end{equation*}
and
\begin{equation*}
\begin{aligned}
\epsilon_{d}\epsilon_{k}\rmH^{k}(\calW_{k, d}\otimes K_{m}, \ZZ_{l}(\frac{k}{2}))_{\frakm_{f}}&\cong \rmH^{1}(K_{m}, \rmH^{k-1}(\calW_{k, d}\otimes\QQ^{\ac}, \ZZ_{l}(\frac{k}{2}))_{\frakm_{f}})\\
&\cong \rmH^{1}(K_{m}, \rmH^{1}(X\otimes\QQ^{\ac}, \calL_{k-2}(\ZZ_{l})(1))_{\frakm_{f}}).\\
\end{aligned}
\end{equation*}
\begin{lemma}\label{v0}
The image of the element $\epsilon_{d}\epsilon_{k}y^{\frac{k-2}{2}}_{m}\in \epsilon_{d}\epsilon_{k}\mathrm{CH}^{\frac{k-2}{2}}(A^{\frac{k-2}{2}}_{m}\otimes K_{m})\otimes\ZZ_{l} $ under the cycle class map to $L_{k-2}(\ZZ_{l})$ can be identified with the vector $\mathbf{v}^{*}_{0}$ up to sign. 
\end{lemma}
\begin{proof}
This follows from the fact that 
\begin{enumerate}
\item $\mathbf{v}^{*}_{0}$ and  $\epsilon_{d}\epsilon_{k}cl(y^{\frac{k-2}{2}}_{m})$ are the eigenvector of the action by $K$ with eigenvalue $1$;
\item $\langle\epsilon_{d}\epsilon_{k}cl(y^{\frac{k-2}{2}}_{m}), \epsilon_{d}\epsilon_{k}cl(y^{\frac{k-2}{2}}_{m})\rangle=\langle\mathbf{v}^{*}_{0}, \mathbf{v}^{*}_{0}\rangle=D^{k-2}_{K}$.
\end{enumerate}
These properties characterize an element in $L_{k-2}(\ZZ_{l})$ up to sign. See \cite[Lemma 7.2]{Chida}.
\end{proof}
Recall the pairing  
\begin{equation*}
\langle\hphantom{a},\hphantom{a}\rangle_{B}:S^{B}_{k}(N^{+}, \calO)\times S^{B}_{k}(N^{+}, \calO)\rightarrow \calO
\end{equation*}
defined as in  \eqref{pairing}. It induces a pairing 
\begin{equation*}
\langle\hphantom{a},\hphantom{a}\rangle_{B}:S^{B}_{k}(N^{+}, \calO)/I^{[p]}_{f, n}\times S^{B}_{k}(N^{+}, \calO)[I^{[p]}_{f, n}]\rightarrow \calO_{n}.
\end{equation*}

\begin{theorem}[Second reciprocity law] \label{second-law}
Let $p$ be an $n$-admissible prime for $f$ and assume that $\bar{\rho}_{f, \lambda}$ satisfies the assumption $(\mathrm{CR}^{\star})$. Let $f^{[p]}_{\pi^{\prime}}$ be $l$-adically normalized and $\widehat{f}^{[p]}_{\pi^{\prime}, n}$ be a generator of $S^{B}_{k}(N^{+}, \calO)[I^{[p]}_{f, n}]$, then we have the following relation between the Heegner cycle class of level  $m$ $\kappa_{n}(m)$ and the theta element $\Theta_{m}(f^{[p]}_{\pi^{\prime}})$
\begin{equation*}
\sum_{\sigma\in \mathcal{G}_{m}}\langle \mathrm{loc}_{p} (\sigma\cdot\kappa_{n}(m)), \widehat{f}^{[p]}_{\pi^{\prime}, n}\rangle_{B}=u\cdot\Theta_{m}(f^{[p]}_{\pi^{\prime}}) \mod \varpi^{n}
\end{equation*}
where $u\in \calO_{n}$ is a unit. 
\end{theorem}
\begin{proof}
Let $\mathbf{1}^{[\mathbf{v}^{*}_{0}]}_{\sigma(x_{m}(1))\cdot\tau^{N^{+}}}$ be the characteristic function of the point $\sigma(x_{m}(1))\cdot\tau^{N^{+}}$ in $X^{B}$ with $\sigma\in \mathcal{G}_{m}$. Note that $\mathbf{1}^{[\mathbf{v}^{*}_{0}]}_{\sigma(x_{m}(1))\cdot\tau^{N^{+}}}$ gives an element of $S^{B}_{k}(N^{+}, \calO)$. It follows by Lemma \ref{v0} and the definition of $\Phi_{n}$, the element $\mathrm{loc}_{p} (\sigma\cdot\kappa_{n}(m))$ is given by the element $\mathbf{1}^{[\mathbf{v}^{*}_{0}]}_{\sigma(x_{m}(1))\cdot\tau^{N^{+}}}[\sigma]$. Therefore we have the following equation
\begin{equation*}
\begin{aligned}
\sum_{\sigma\in \mathcal{G}_{m}}\langle\mathrm{loc}_{p} (\sigma\cdot\kappa_{n}(m)), \widehat{f}^{[p]}_{\pi^{\prime}, n}\rangle_{B}&=\sum_{\sigma\in \mathcal{G}_{m}}\langle \mathbf{1}^{[\mathbf{v}^{*}_{0}]}_{\sigma(x_{m}(1))\cdot \tau^{N^{+}}}, \widehat{f}^{[p]}_{\pi^{\prime}, n}\rangle_{B}[\sigma]\\
&=\sum_{\sigma\in \mathcal{G}_{m}}\langle \mathbf{v}^{*}_{0}, \widehat{f}^{[p]}_{\pi^{\prime}, n}(\sigma\cdot x_{m}(1))\rangle_{k}[\sigma]\\
&=u\cdot\Theta_{m}(f^{[p]}_{\pi^{\prime}}) \mod \varpi^{n}.
\end{aligned}
\end{equation*}
\end{proof}

Next, let $(p, p^{\prime})$ be a pair of $n$-admissible primes for $f$. Then we can consider the Shimura curves $X^{\prime\prime}$ and $X^{\prime\prime}_{d}$ and the corresponding Kuga-Sato varieties $\calW^{\prime\prime}_{k, d}$ defined in \S 2.6. Note that they correspond to the indefinite quaternion algebra $B^{\prime\prime}$ with discriminant $pp^{\prime}N^{-}$. We can define in the same manner as in \eqref{Heeg-pt} the Heegner point 
\begin{equation}
{P}^{[pp^{\prime}]}_{m}(a)=[z^{\prime}, a^{(p^{\prime})}\varsigma\tau^{N^{+}}]_{\CC}\in X^{\prime\prime}(K_{m})
\end{equation}
for $a\in \widehat{K}^{\times}$. Using these points, we can define Heegner cycles 
\begin{equation}
\epsilon_{d}Y^{[pp^{\prime}]}_{m, k} \in  \epsilon_{d}\epsilon_{k}\CH^{\frac{k}{2}}(\calW^{\prime\prime}_{k, d}\otimes K_{m})\otimes\ZZ_{l}
\end{equation}
similarly as in 
Since $(p, p^{\prime})$ are $n$-admissible primes for $f$, there is a homomorphism $\phi^{[pp^{\prime}]}_{f,n}: \TT^{[pp^{\prime}]}\rightarrow \calO_{n}$  such that $\phi^{[pp^{\prime}]}_{f, n}$ agrees with $\phi_{f, n}$ at all Hecke operators away from $pp^{\prime}$ and sends $(U_{p}, U_{p^{\prime}})$ to $(\epsilon_{p}p^{\frac{k-2}{2}}, \epsilon_{p^{\prime}}p^{\prime \frac{k-2}{2}})$. Recall that $I^{[pp^{\prime}]}_{f, n}$ is kernel of $\phi^{[pp^{\prime}]}_{f, n}$. 

\begin{lemma} We have the following statements.
\begin{enumerate}
\item The morphism $\phi^{[pp^{\prime}]}_{f, n}$ can be lifted to a genuine modular form $f^{[pp^{\prime}]}\in S^{\new}_{k}(pp^{\prime}N)$;
\item There is an isomorphism $\rmH^{1}(X^{\prime\prime}_{\QQ^{\ac}}, \calL_{k-2}(\calO)(1))/I^{[pp^{\prime}]}_{f, n}\cong \rmT_{f, n}$.
\end{enumerate}
\end{lemma}
\begin{proof}
It again follows from the main results of \cite{DT} and \cite{DT1} that the morphism $\phi^{[pp^{\prime}]}_{f, n}$ can be lifted to a genuine modular form which we will denote as $f^{[pp^{\prime}]}$. 

To prove the second statement, it follows from the main result of \cite{BLR} that 
\begin{equation*}
\rmH^{1}(X^{\prime\prime}_{\QQ^{\ac}}, \calL_{k-2}(\calO)(1))/I^{[pp^{\prime}]}_{f, n}\cong \rmT^{r}_{f, n} 
\end{equation*}
for some integer $r$. Then one can consider the weight spectral sequence converges to \eqref{Xprimeprime} and use the fact that $S^{B}_{k}(N^{+},\calO)_{/I^{[p]}_{f, n}}$ is of rank one to conclude that $r=1$. 
\end{proof}

Using the above Lemma, we can define the Abel-Jacobi map 
\begin{equation}
\mathrm{AJ}^{[pp^{\prime}]}_{k, n}: \epsilon_{d}\epsilon_{k}\CH^{\frac{k}{2}}(\calW^{\prime\prime}_{k, d}\otimes K_{m})\otimes\ZZ_{l} \rightarrow  \rmH^{1}(K_{m}, \mathrm{T}_{f, n})
\end{equation}
following the same recipe for defining \eqref{AJ-T}. We can define the corresponding Heegner cycle class of level $m$
\begin{equation}
\kappa^{[pp^{\prime}]}_{n}(m)=\mathrm{AJ}^{[pp^{\prime}]}_{k, n}(\epsilon_{d}Y^{[pp^{\prime}]}_{k, d})\in \rmH^{1}(K_{m}, \mathrm{T}_{f, n}).
\end{equation}
Similarly, we define the class
\begin{equation}
\kappa^{[pp^{\prime}]}_{n}=\mathrm{Cor}_{K_{1}/K}\kappa^{[pp^{\prime}]}_{n}(1) \in  \rmH^{1}(K, \mathrm{T}_{f, n}).
\end{equation}
By \cite[Lemma 2.4 and 2.5]{BD-Main} and \cite[Lemma 1.4]{CH-2}, we have an isomorphism 
\begin{equation*}
\rmH^{1}_{\sing}(K_{m, p}, \rmT_{f, n})\cong \rmH^{1}_{\sing}(K_{p}, \rmT_{f, n})\otimes \calO_{n}[\mathcal{G}_{m}]
\end{equation*}
The element  $\partial_{p^{\prime}} \mathrm{loc}_{p^{\prime}}(\kappa^{[pp^{\prime}]}(m))\in  \rmH^{1}_{\sing}(K_{p^{\prime}}, \rmT_{f, n})\otimes \calO_{n}[\mathcal{G}_{m}]$ under the composite below
\begin{equation*}
\rmH^{1}(K_{m}, \rmT_{f, n}) \xrightarrow{\mathrm{loc}_{p^{\prime}}}\rmH^{1}(K_{m, p^{\prime}}, \rmT_{f, n})\xrightarrow{\partial_{p^{\prime}}} \rmH^{1}_{\sing}(K_{m, p^{\prime}}, \rmT_{f, n})
\end{equation*}
can be considered as an element in $S^{B}(N^{+}, \calO)/I^{[p]}_{f, n}\otimes \calO_{n}[\mathcal{G}_{m}]$ using the map $\Xi_{n}$ given by Theorem \ref{level-raise-curve}.

\begin{theorem}[First reciprocity law] \label{First-Law}
Let $(p, p^{\prime})$ be a pair of $n$-admissible prime for $f$ and assume that $\bar{\rho}_{f, \lambda}$ satisfies assumption $(\mathrm{CR}^{\star})$. Let $f^{[p]}_{\pi^{\prime}}$ be $l$-adically normalized and $\widehat{f}^{[p]}_{\pi^{\prime}, n}$ be a generator of $S^{B}_{k}(N^{+}, \calO)[I^{[p]}_{f, n}]$,  then we have the following relation between the Heegner cycle class $\kappa^{[pp^{\prime}]}_{n}(m)$ and the theta element $\Theta_{m}(f^{[p]}_{\pi^{\prime}})$
\begin{equation}
\sum_{\sigma\in \mathcal{G}_{m}}\langle\partial_{p^{\prime}}\mathrm{loc}_{p^{\prime}} (\sigma\cdot\kappa^{[pp^{\prime}]}_{n}(m)), \widehat{f}^{[p]}_{\pi^{\prime}, n}\rangle_{B}=u\cdot\Theta_{m}(f^{[p]}_{\pi^{\prime}}) \mod \varpi^{n}
\end{equation}
where $u\in \calO^{\times}_{n}$ is a unit. 
\end{theorem}
\begin{remark}
This theorem is proved by \cite{Chida} in a slightly different set-up, but for completeness we sketch a proof.
\end{remark}
\begin{proof}
By the proof of Theorem \ref{level-raise-curve},  $\partial_{p^{\prime}}\mathrm{loc}_{p^{\prime}} (\sigma\cdot\kappa^{[pp^{\prime}]}_{n}(m))$ is given by $\mathbf{1}^{[\mathbf{v}^{*}_{0}]}_{\sigma(x_{m}(1))\cdot\tau^{N^{+}}}[\sigma]$. Therefore we have the following equation
\begin{equation*}
\begin{aligned}
\sum_{\sigma\in \mathcal{G}_{m}}\langle\partial_{p^{\prime}}\mathrm{loc}_{p^{\prime}} (\sigma\cdot\kappa^{[pp^{\prime}]}_{n}(m)), \widehat{f}^{[p]}_{\pi^{\prime}, n}\rangle_{B}&=\sum_{\sigma\in \mathcal{G}_{m}}\langle \mathbf{1}^{[\mathbf{v}^{*}_{0}]}_{\sigma(x_{m}(1))\cdot \tau^{N^{+}}}, \widehat{f}^{[p]}_{\pi^{\prime}, n}\rangle_{B}[\sigma]\\
&=\sum_{\sigma\in \mathcal{G}_{m}}\langle \mathbf{v}^{*}_{0}, \widehat{f}^{[p]}_{\pi^{\prime}, n}(\sigma\cdot x_{m}(1))\rangle_{k}[\sigma]\\
&=u\cdot\Theta_{m}(f^{[p]}_{\pi^{\prime}}) \mod \varpi^{n}.
\end{aligned}
\end{equation*}
\end{proof}

\section{Converse to Gross-Zagier-Kolyvagin type theorem}
\subsection{Selmer groups of modular forms} Recall that $f$ is a modular form of weight $k$ level $\Gamma_{0}(N)$ such that $N=N^{+}N^{-}$. We assume that 
\begin{equation*}\tag{Heeg}
\text{\emph{$N^{-}$ is square free and consists of even number of prime factors that are inert in $K$}}.
\end{equation*}
Let $K$ be an imaginary quadratic field with discriminant $-D_{K}$ such that $(D_{K}, N)=1$. Let $l>2$ be a prime such that $l\nmid ND_{K}$. Recall $\rho_{f, \lambda}: G_{\QQ}\rightarrow \GL_{2}(E_{\lambda})$ is the $\lambda$-adic Galois representation attached to the form $f$ which is characterized by the fact that the trace of Frobenius at $p\nmid N$ agrees with $a_{p}(f)$ and the determinant of $\rho_{f, \lambda}$ is $\epsilon_{l}^{k-1}$ with $\epsilon_{l}$ the $l$-adic cyclotomic character. Recall we are interested in the twist $\rho^{*}_{f, \lambda}=\rho_{f, \lambda}(\frac{2-k}{2})$. Let $\rmV_{f, \lambda}$ be the representation space for $\rho^{*}_{f, \lambda}$.  We normalize the construction of $\rho_{f, \lambda}$ such that it occurs in the cohomology $\rmH^{1}(X_{\QQ^{\ac}}, \calL_{k-2}(E_{\lambda})(\frac{k}{2}))$ and therefore  $\rho^{*}_{f, \lambda}$ occurs in the cohomology $\rmH^{1}(X_{\QQ^{\ac}}, \calL_{k-2}(E_{\lambda})(1))$. The modular form $f$ gives rise to a homomorphism $\phi_{f}: \TT\rightarrow \calO$ corresponding to the Hecke eigen-system of $f$. Let $n\geq 1$, we have $\phi_{f, n}: \TT\rightarrow \calO_{n}$ the natural reduction of $\phi_{f}$ by $\varpi^{n}$.  We define $I_{f, n}$ to be the kernel of the morphism $\phi_{f, n}$ and $\frakm_{f}$ the unique maximal ideal of $\TT$ containing $I_{f, n}$. We choose a $G_{\QQ}$-stable lattice $\rmT_{f, \lambda}$ in $\rmV_{f, \lambda}$ and denote by $\rmT_{f, n}$ the reduction $\rmT_{f,\lambda}/\varpi^{n}$. We also recall that the residual Galois representation $\bar{\rho}_{f, \lambda}$ satisfies the assumption $(\mathrm{CR}^{\star})$. In light of this assumption and Theorem \ref{level-raise-curve} $(3)$, we can choose the lattice $\rmT_{f, \lambda}$ to be $\rmH^{1}(X_{\QQ^{\ac}}, \calL_{k-2}(\calO))_{\frakm_{f}}$. We denote by $\rmA_{f, \lambda}$ the divisible module given by $\rmV_{f, \lambda}/\rmT_{f, \lambda}$. Then we set 
\begin{equation}\label{div-mod}
\rmA_{f, n}=\ker[\rmA_{f,\lambda}\xrightarrow{\varpi^{n}}\rmA_{f, \lambda}].
\end{equation}
Note that $\rmA_{f, n}$ is the Kummer dual of $\rmT_{f, n}$. 

Let $\rmM= \rmT_{f, n}\text{ or }\rmA_{f,n}$ and $v\mid N^{-}$, then let $F^{+}_{v}\rmM$ be the unique line of $\rmM$ such that $G_{\QQ_{v}}$ acts by $\chi_{v}\tau_{v}$ with $\tau_{v}$ the non-trivial unramified character of $G_{\QQ_{v}}$. Then we define the ordinary part of $\rmH^{1}(K_{v}, \rmM)$ to be 
\begin{equation*}
\rmH^{1}_{\mathrm{ord}}(K_{v}, \rmM)= \ker[\rmH^{1}(K_{v}, M)\rightarrow \rmH^{1}(K_{v}, \rmM/F^{+}_{v}\rmM)].
\end{equation*}
Let $p\nmid N$ be an $n$-admissible prime for $f$, then we set $F^{+}_{p}\rmM$ to be the unique line such that $\mathrm{Frob}_{p}$ acts by $\epsilon_{p}p$ and $F^{-}_{p}\rmM$ be the line such that $\mathrm{Frob}_{p}$ acts by $\epsilon_{p}$ then
\begin{equation}
\begin{aligned}
\rmH^{1}(K_{p}, \rmM)&=\rmH^{1}(K_{p}, F^{-}_{p}\rmM)\oplus \rmH^{1}(K_{p}, F^{+}_{p}\rmM)\\
&\cong \rmH^{1}_{\mathrm{fin}}(K_{p}, \rmM)\oplus \rmH^{1}_{\mathrm{ord}}(K_{p}, \rmM).
\end{aligned}
\end{equation}
In order to apply the results from Iwasawa theory, we assume that $f$ is \emph{$l$-ordinary}. If $v\mid l$ in $K$, let $F^{+}_{v}\rmM$ be the  unique line such that $G_{\QQ_{l}}$ acts by $\epsilon^{\frac{k}{2}}_{l}$.  We define 
\begin{equation}
\rmH^{1}_{\mathrm{ord}}(K_{v}, \rmM)= \ker[\rmH^{1}(K_{v}, M)\rightarrow \rmH^{1}(K_{v}, \rmM/F^{+}_{v}\rmM)].
\end{equation}

Following the notation in \cite{How}, we define the local conditions $\calF^{c}_{b}(a)$ for a triple of integers $(a, b, c)$ and $l$ by
\begin{equation}\label{Selmer}
\rmH^{1}_{\calF^{a}_{b}(c)}(K_{v}, \rmM)=
\begin{cases}
\rmH^{1}_{\mathrm{fin}}(K_{v}, \rmM) &\hphantom{a}\text{if}\hphantom{a} v\nmid abcl \\
\rmH^{1}(K_{v}, M) &\hphantom{a}\text{if}\hphantom{a} v\mid a\\
0 &\hphantom{a}\text{if}\hphantom{a} v\mid b\\
\rmH^{1}_{\mathrm{ord}}(K_{v}, \rmM) &\hphantom{a}\text{if}\hphantom{a} v\mid c\\
\rmH^{1}_{\mathrm{ord}}(K_{v}, \rmM) &\hphantom{a}\text{if}\hphantom{a} v\mid l \\
\end{cases}
\end{equation}
In other words, at places dividing $a$, we use the relaxed local condition; at places dividing $a$, we use the relaxed local condition  If any of $(a, b, c)$ is $1$, then we omit it from the notation. We define the Selmer group for $\rmM$ by
\begin{equation*}
\mathrm{Sel}_{\calF^{a}_{b}(c)}(K, \rmM)=\{s\in \rmH^{1}(K, \rmM): \mathrm{loc}_{v}(s)\in \rmH^{1}_{\calF^{a}_{b}(c)}(K_{v}, \rmM)\hphantom{a}\text{for all}\hphantom{a}v\}. 
\end{equation*}
In this article, we will be mainly concerned with the Selmer group $\mathrm{Sel}_{\calF(N^{-})}(K, \rmM)$. Notice that the Abel-Jacobi map
\begin{equation*}
\mathrm{AJ}_{k, n}: \epsilon_{d}\epsilon_{k}\CH^{\frac{k}{2}}(\calW_{k, d}\otimes K)\otimes\ZZ_{l} \rightarrow  \rmH^{1}(K, \mathrm{T}_{f, n})
\end{equation*}
factors through  $\mathrm{Sel}_{\calF(N^{-})}(K, \rmT_{f, n})$.  This is well-known except for a justification for primes dividing $N^{-}$. Let $v\mid N^{-}$ and suppose that $\bar{\rho}_{f, \lambda}$ is ramified. Then it follows from our assumption $(\mathrm{CR}^{\star})$ that $v\not \equiv 1\mod l$ and a simple calculation using \cite[Theorem 2.17]{DDT} shows that $\vert \rmH^{1}(K_{v}, \rmT_{f, n})\vert=\vert \rmT^{G_{K_{v}}}_{f, n}\vert^{2}=0$.  Let $v\mid N^{-}$ and suppose that $\bar{\rho}_{f, \lambda}$ is unramified at $v$, then we have a decomposition 
\begin{equation*}
\rmH^{1}(K_{v}, \rmT_{f, n})=\rmH^{1}_{\mathrm{ord}}(K_{v}, \rmT_{f, n})\oplus \rmH^{1}_{\mathrm{fin}}(K_{v}, \rmT_{f, n}).
\end{equation*}
Then our claim follows from the proof of  the ramified level raising in Theorem \ref{level-raise-curve-ram}. 

\begin{corollary}
The theta elements of Chida-Hsieh defined in \eqref{Theta} and the Heegner cycle classes defined in \eqref{Heegner-cycle} form a bipartite Euler system of odd type for the Selmer structures given by $\calF(N^{-})$ over $K$.
\end{corollary}
\begin{proof}
Recall the definition of a bipartite Euler system in \cite[Definition 2.3.2]{How}. This follows from the First reciprocity law \ref{First-Law} and the Second reciprocity law \ref{second-law} proved before. 
\end{proof}

\subsection{The proof of the main result} 
Now we can state and prove the main result of this article.
\begin{theorem}\label{main-thm}
Suppose $(f, K)$ is a pair that satisfies the generalized Heegner hypothesis $(\mathrm{Heeg})$ and $f$ is ordinary at $l$. Assume that $\bar{\rho}_{f, \lambda}$ satisfies the hypothesis $(\mathrm{CR}^{\star})$. If $\mathrm{Sel}_{\calF(N^{-})}(K, \rmT_{f, 1})$ is of dimension $1$ over $\FF_{\lambda}$, then the class $\kappa_{1}$ is non-zero in  $\mathrm{Sel}_{\calF(N^{-})}(K, \rmT_{f, 1})$.
\end{theorem}
\begin{remark}
The above theorem can be considered as a generalization of the converse to Gross-Zagier-Kolyvagin type theorem proved by Wei Zhang \cite{wei-zhang} and Skinner \cite{Skinner} to the higher weight case. 
\end{remark}

Let $p$ be an $1$-admissible prime for $f$ and let $f^{[p]}$ be the level raising of the modular form $f$ constructed in Theorem \ref{level-raise-curve}. Since the residual representation of $f^{[p]}$ and $f$ are isomorphic, we can regard $\mathrm{Sel}_{\calF_{(pN^{-})}}(K, \rmT_{f, 1})$ as the residual Selmer group for $f^{[p]}$.  Then we have the following result concerning the Selmer group of $f$ and $f^{[p]}$. 
\begin{proposition}\label{rank-lowering}
Suppose that $\mathrm{loc}_{p}: \mathrm{Sel}_{\calF(N^{-})}(K, \rmT_{f, 1})\rightarrow \rmH^{1}_{\mathrm{fin}}(K_{p}, \rmT_{f, 1})$ is surjective (equivalently non-trivial). Then we have
\begin{equation*}
\dim_{k} \mathrm{Sel}_{\calF(N^{-})}(K, \rmT_{f, 1}) = \dim_{k} \mathrm{Sel}_{\calF(pN^{-})}(K, \rmT_{f, 1})+1.
\end{equation*}
In this case, we have 
\begin{equation*}
\mathrm{Sel}_{\calF(N^{-})}(K, \rmT_{f, 1})= \mathrm{Sel}_{\calF^{p}(N^{-})}(K, \rmT_{f, 1}), \hphantom{a} \mathrm{Sel}_{\calF(pN^{-})}(K, \rmT_{f, 1})= \mathrm{Sel}_{\calF_{p}(N^{-})}(K, \rmT_{f, 1}).
\end{equation*}
\end{proposition}
\begin{proof}
This follows from \cite[Proposition 2.2.9, Corollary 2.2.10]{How}. More precisely,  we have the following cartesian diagram of Selmer structures 
\begin{equation}\label{picard-lef}
\begin{tikzcd}
\mathrm{Sel}_{\calF^{p}(N^{-})}  & \mathrm{Sel}_{\calF(N^{-})}\arrow[l, "x"]\\
\mathrm{Sel}_{\calF(pN^{-})}   \arrow[u, "y"]             & \mathrm{Sel}_{\calF_{p}(N^{-})} .\arrow[l,"x"]\arrow[u, "y"] \\
\end{tikzcd}
\end{equation}
Here, the labels $x$ and $y$ on the arrows stand for the length of the respective quotients. We have $x+y=1$ by  \cite[Proposition 2.2.9]{How} . Since $p$ is $1$-admissible, the local conditions $\rmH^{1}_{\ord}(K_{p}, \rmT_{f,1})$ and $\rmH^{1}_{\mathrm{fin}}(K_{p}, \rmT_{f,1})$ are dual to each other under the local Tate duality. Therefore if 
\begin{equation*}
\mathrm{loc}_{p}: \mathrm{Sel}_{\calF(N^{-})}(K, \rmT_{f, 1})\rightarrow \rmH^{1}_{\mathrm{fin}}(K_{p}, \rmT_{f, 1})
\end{equation*}
is surjective, then $y=1$ and $x=0$. 
\end{proof}

Next we combine results from \cite{CH-2} and \cite{SU} to deduce a special value formula for the modular form $f^{[p]}$.  For this, let 
$\mathrm{Sel}(K, A_{f^{[p]}})=\varinjlim_{n}\mathrm{Sel}(K, A_{f^{[p]}, n})$
be the minimal Selmer group of $f^{[p]}$ defined as in \cite[Introduction]{CH-2}. Here $A_{f^{[p]}}$ and $A_{f^{[p]}, n}$ are defined the exact same way as in 
\eqref{div-mod}. We will also use the Selmer group $\mathrm{Sel}_{pN^{-}}(K, A_{f^{[p]}, n})$ defined in  \cite[Definition 1.2]{CH-2}.

\begin{theorem}\label{rank-0}
Suppose $(f, K)$ is a pair that satisfies the generalized Heegner hypothesis $(\mathrm{Heeg})$. Assume that $\bar{\rho}_{f, \lambda}$ satisfies the hypothesis $(\mathrm{CR}^{\star})$ and in addition assume that $f$ is $l$-ordinary. Then $L(f^{[p]}/K, 1)\neq 0$ if and only if $\mathrm{Sel}(K, A_{f^{[p]}})$ is finite and we have
\begin{equation*}
{v}_{\varpi}(\frac{L(f^{[p]}/K, 1)}{\Omega^{\mathrm{can}}_{f^{[p]}}})=\mathrm{leng}_{\calO} \mathrm{Sel}(K, A_{f^{[p]}}) + \sum_{q\mid pN}t_{q}(f).
\end{equation*}
\end{theorem} 
\begin{proof}
Since $\bar{\rho}_{f, \lambda}$ satisfies the hypothesis $(\mathrm{CR}^{\star})$, the form $f^{[p]}$ satisfies the hypothesis $(\mathrm{CR}^{+})$ of \cite{CH-2}. Therefore we can combine \cite[Corollary 2]{CH-2} and the main result of \cite{SU} to obtain the following equation
\begin{equation*}
{v}_{\varpi}(\frac{L(f^{[p]}/K, 1)}{\Omega_{f^{[p]}, pN^{-}}})=\mathrm{leng}_{\calO} \mathrm{Sel}(K, A_{f^{[p]}}) + \sum_{q\mid N^{+}}t_{q}(f^{[p]}).
\end{equation*}
It follows from \cite[Corollary 5.8]{OK} that 
\begin{equation*}
{v}_{\varpi}(\frac{\Omega^{\mathrm{can}}_{f^{[p]}}}{\Omega_{f^{[p]}, pN^{-}}})=\sum_{q\mid pN^{-}}t_{q}(f^{[p]}).
\end{equation*}
The result follows.
\end{proof}
\begin{remark}
Instead of using the one-sided divisibility of Chida-Hsieh \cite{CH-2}, one can apply the main result of \cite{Kato} to $f$ and its quadratic twist $f^{K}$ to get the same result. This is the approach used in \cite{wei-zhang}.
\end{remark}

\begin{myproof}{Theorem}{\ref{main-thm}}
Suppose $c$ is a generator of $\mathrm{Sel}_{\calF(N^{-})}(K, \rmT_{f, 1})$. Then we can find an $1$-admissible prime $p$ for $f$ such that $\mathrm{loc}_{p}(c)\in \rmH^{1}_{\mathrm{fin}}(K_{p},\rmT_{f, 1})$ is non-zero by using the same proof of \cite[Theorem 6.3]{CH-2}. Then Proposition \ref{rank-lowering} implies that 
\begin{equation*}
\dim_{k} \mathrm{Sel}_{\calF{(pN^{-})}}(K, \rmT_{f, 1})=\dim_{k} \mathrm{Sel}_{\calF{(N^{-})}}(K, \rmT_{f, 1})-1=0.
\end{equation*}
Since $\mathrm{Sel}_{pN^{-}}(K, \rmT_{f^{[p]}, 1})$ can be regarded as a subspace of $\mathrm{Sel}_{\calF(N^{-})}(K, \rmT_{f, 1})$, we know that 
\begin{equation*}
\mathrm{Sel}_{pN^{-}}(K, \rmA_{f^{[p]}, 1})=0. 
\end{equation*}
Then by the control theorem of \cite[Proposition 1.9(2)]{CH-2},  we have $\mathrm{Sel}_{pN^{-}}(K, \rmA_{f^{[p]}})=0$. Therefore $\mathrm{Sel}(K, \rmA_{f^{[p]}})=0$ and $\sum_{q\mid N^{+}}t_{q}(f^{[p]})=\sum_{q\mid N^{+}}t_{q}(f)=0$, by the proof of \cite[Corollary 6.15]{CH-2}. Then we can apply Theorem \ref{rank-0} and conclude that
\begin{equation*} 
{v}_{\varpi}(\frac{L(f^{[p]}/K, 1)}{\Omega_{f^{[p]}, N^{-}}})=0. 
\end{equation*}
The Second reciprocity law in Theorem \ref{second-law} and specialization formula for the theta element in Theorem \ref{special-value} allows us to conclude that $\mathrm{loc}_{p}(\kappa_{1})$ is non-zero in $\rmH^{1}_{\mathrm{fin}}(K_{p}, \rmT_{f, 1})$. Therefore $\kappa_{1}$ is non-zero in $\mathrm{Sel}_{\calF(N^{-})}(K, \rmT_{f, 1})$ and we are done.
\end{myproof}

\end{document}